\documentclass{siamart1116}
\usepackage{array}
\usepackage{tikz}
\usepackage{color}
\usepackage{pgfplots}
\usepackage{pgfplotstable}
\usepackage{booktabs}
\usepackage{esint}
\usepackage{listings}
\usepackage{url}
\usepackage{amsmath,amssymb,amsfonts,cite}
\usetikzlibrary{patterns}
\usepackage{enumerate}
\usepackage{algorithm}
\usepackage{algpseudocode}
\usepackage{footnote}
\usepackage{cleveref}
\makesavenoteenv{table}
\makesavenoteenv{tabular}
\renewcommand{\tilde}{\widetilde}

\colorlet{DenseBlockColor}{gray!60}

\newcommand{\trace}{\mathrm{tr}}
\crefformat{footnote}{#2\footnotemark[#1]#3}

\pgfplotsset{compat=1.14}

\newsiamremark{remark}{Remark}
\newsiamremark{assumption}{Assumption}

\DeclareMathOperator{\diag}{diag}
\DeclareMathOperator{\vect}{vec}

\newcommand{\norm}[1]{\lVert#1\rVert}
\definecolor{mygreen}{RGB}{28,172,0} 
\definecolor{mylilas}{RGB}{170,55,241}
\definecolor{stringcolor}{RGB}{180,10,10}
\definecolor{mygray}{RGB}{240,240,240}
\definecolor{mygray2}{RGB}{200,200,200}
\author{
	Stefano Massei\thanks{TU/e Eindhoven, Netherlands,
		\email{s.massei@tue.nl}.} \and
	Leonardo Robol\thanks{Department of Mathematics, University of Pisa, 
		\email{leonardo.robol@unipi.it}.}
}

\title{Mixed precision recursive block diagonalization  for bivariate functions of matrices\thanks{The authors are members of the research group
                INdAM--GNCS. This work has been partially supported by
                the GNCS project ``Metodi low-rank per problemi di algebra lineare
                con struttura data-sparse''.}}

\begin{document}
\maketitle
\begin{abstract}
Various numerical linear algebra problems can be formulated
as evaluating bivariate function of matrices. The most
 notable examples are the Fr\'echet
derivative along a direction, the evaluation of 
(univariate) functions of Kronecker-sum-structured matrices
and the solution of Sylvester matrix equations. In this 
work, we propose a recursive block diagonalization algorithm 
for computing bivariate functions of matrices of small to
medium size, for which dense liner algebra is appropriate. 
The algorithm combines a blocking strategy, as in the 
Schur-Parlett scheme, and an evaluation procedure for the
diagonal blocks. We discuss two implementations of the latter. 
The first is a natural choice based on Taylor expansions, whereas the second
is derivative-free and relies on a multiprecision 
perturb-and-diagonalize approach. In particular, the
appropriate use of multiprecision guarantees backward stability
without affecting the efficiency in the generic case. This makes the second approach more robust. The whole
method has cubic complexity and it is closely related
to the well-known Bartels-Stewart algorithm for Sylvester 
matrix equations when applied to $f(x,y)=\frac{1}{x+y}$. We validate the performances of the proposed numerical method  on several problems with different conditioning properties. 
	\end{abstract}
\section{Introduction}
Matrix functions \cite{higham2008functions} such as the matrix inverse, the matrix exponential, the matrix square root and many others, arise in an endless list of applications including analysis of complex networks~\cite{estrada2010network}, signal processing~\cite{hjorungnes2011complex}, solution of ODEs~\cite{hochbruck2010exponential} and control theory~\cite{antoulas2005approximation}. The  practical computation of univariate functions of matrices has been intensively analyzed from different angles such as the reduction to triangular form
\cite{davies2003schur}, polynomial and rational approximants~\cite{higham2009scaling}, contour integrals~\cite{hale2008computing} and projection on low dimensional subspaces~\cite{guttel2010rational}.

 The matrix function concept extends quite naturally to the bivariate setting. Given  two  square matrices $A\in\mathbb C^{m\times m}$, $B\in\mathbb C^{n\times n}$  and a complex-valued function $f(x,y)$, the
\emph{bivariate matrix function} $f\{A,B\}$ \cite{kressner2011bivariate} is a linear operator on $\mathbb C^{m\times n}$. As in the univariate case, the  definition of $f\{A,B\}$ can be given, equivalently, in terms of (bivariate) Hermite interpolation, power series expansion and contour integration. We report the latter formulation which is the most useful for our work. Let $\Lambda_A$ and $\Lambda_B$ be the spectra of $A$ and $B$, respectively, and let $f(x,y)$ be analytic in an open neighborhood of $\Lambda_A\times\Lambda_B$; $f\{A,B\}$ is defined as
\begin{align}
f\{A,B\}:\mathbb C^{m\times n}&\longrightarrow\  \mathbb C^{m\times n}\nonumber\\
C\qquad&\longrightarrow\  f\{A,B\}(C):=\oint_{\Gamma_A}\oint_{\Gamma_B}f(x,y)(xI-A)^{-1}C(yI-B^T)^{-1}\ dxdy,\label{eq:cauchy-def}
\end{align}
with $\Gamma_A,\Gamma_B$ closed contours enclosing $\Lambda_A$ and $\Lambda_B$, respectively. 

Although their usual formulations involve different frameworks, the following popular linear algebra problems 
correspond to evaluate a bivariate matrix function:
\begin{enumerate}
	\item The solution of the \emph{Sylvester equation}
	$$
	AX+XB = C
	$$
is given by $X=f_1\{A,B\}(C)$ where $f_1(x,y)=\frac{1}{x+y}$.
		\item Given a univariate matrix function $g(A)$, the \emph{Fr\'echet derivative} of $g$ at $A$ in the direction $C$, i.e. $Dg\{A\}(C):=\lim_{t\to 0}\frac 1t(g(A+tC)-g(A))$, verifies $Dg\{A\}(C)=f_2\{A,A^T\}(C)$ where $f_2(x,y)$ is the finite difference quotient
		$$
		f_2(x,y)=\begin{cases}
		\frac{g(x)-g(y)}{x-y}&x\neq y\\
		g'(x) & x=y
		\end{cases}.
		$$
	\item Given a univariate function $h(x)$, a matrix with the Kronecker sum structure $\mathcal A= A\otimes I+I\otimes B$ and a vector $v$, we have that $$h(\mathcal A)v=\vect\left( f_3\{A,B\}(C) \right),\qquad f_3(x,y)=h(x+y),$$
	where the matrix $C$ verifies $\vect(C)=v$.
\end{enumerate}
We stress that effective algorithms specialized for each of these case studies, or their subcases, already exist; see \cite{simoncini2016computational} for Sylvester equations, \cite{al2009computing,al2013computing} for the Frech\'et derivative, and \cite{benzi2017approximation,massei2020rational} for functions of Kronecker sums. Quite recently,
 computing the application of a generic bivariate matrix function to a low-rank matrix $C$ has been addressed in \cite{kressner2019krylov} and a multivariate version of the Crouzeix-Palencia bound \cite{crouzeix2017numerical} has been proved in \cite{crouzeix2020bivariate}.
The ultimate goal of this work is to provide an algorithm for the computation of $f\{A,B\}(C)$ in the most general scenario, i.e. only requiring that the matrices $A,B,C$ have appropriate sizes and that $f$ is analytic on the Cartesian product of the spectra of $A$ and $B$.
 
  \subsection{Diagonalization of A and/or B}
 
  If at least one  among $A$ and $B$ is a normal matrix or has a well conditioned eigenvector matrix, then evaluating $f\{A,B\}(C)$ simplifies considerably. 
  
  Note that, if $D_A:=\diag(\lambda_1^{A},\dots,\lambda_m^{A})$ and $D_B:=\diag(\lambda_1^{B},\dots,\lambda_n^{B})$ are diagonal matrices, then $X:=f\{D_A,D_B\}(C)$ is given by $(X)_{ij}=f(\lambda_i^A,\lambda_j^B)\cdot C_{ij}$. In addition, \eqref{eq:cauchy-def} implies the following property that describes the interplay between bivariate matrix functions and similarity transformations: Given invertible matrices $S_A\in\mathbb C
^{m\times m}$ and $S_B\in\mathbb C^{n\times n}$, it holds
\begin{equation}\label{eq:similarity}
    f\{A,B\}(C)=S_A\cdot f\{S_A^{-1}AS_A,S_B^{-1}BS_B\}(S_A^{-1}CS_B^{-T})\cdot S_B^{T}.
\end{equation}
  Therefore, in the case $A=S_AD_AS_A^{-1}$, $B=S_BD_BS_B^{-1}$ for well conditioned $S_A$ and $S_B$, we make use of \eqref{eq:similarity} to get
  \begin{equation}\label{eq:diag-coef}
  \begin{aligned}
  f\{A,B\}(C)&= S_Af\{D_A,D_B\}(\widetilde C)S_B^T   &\widetilde C&:=S_A^{-1}CS_B^{-T},\\
  &= S_A(F\circ \widetilde C)S_B^T,  &(F)_{ij}&:=f(\lambda_i^{A},\lambda_j^B),
  \end{aligned}
  \end{equation}
  where $\circ$ indicates the component-wise Hadamard product of matrices. In particular, only evaluations of $f$ on scalar entries are needed; we call the procedure based on \eqref{eq:diag-coef} \textsc{fun2\_diag} and we report it in Algorithm~\ref{alg:fun2m_diag}.
  \begin{algorithm}[H] 
  	\small 
  	\caption{}\label{alg:fun2m_diag}
  	\begin{algorithmic}[1]
  		\Procedure{fun2\_diag}{$f,A,B,C$}
  		\State $[S_A, D_A]=\texttt{eig}(A)$
  		\State $[S_B, D_B]=\texttt{eig}(B)$
  		\State $\widetilde C \gets S_A^{-1}CS_B^{-T}$
  		\State $F\gets \left(f(\lambda_i^A,\lambda_j^B)\right)_{ij}$ \quad $i=1\dots m$, $j=1,\dots,n$
  		\State\Return $S_A(F\circ\widetilde C)S_B^T$
  		\EndProcedure
  	\end{algorithmic}
  \end{algorithm} 
If only $B$ is diagonalized, then it is convenient to rely on the following formula \cite[Section 5]{kressner2011bivariate}:
  \begin{align} \label{eq:half-diagonalize}
f\{A,B\}(C)&=f\{A,D_B\}(\widetilde C)S_B^T & \widetilde C:=CS_B^{-T}=[\widetilde c_1|\dots|\widetilde c_n],\\
&=[f_{\lambda_1^B}(A)\widetilde c_1|\dots|f_{\lambda_n^B}(A)\widetilde c_n],&f_{\lambda_j^B}(x):=f(x, \lambda_j^B). \nonumber
  \end{align}
  Since the previous expression only involves the evaluation of univariate matrix functions, it is performed via the Schur-Parlett algorithm \cite{davies2003schur}, which is implemented, for instance, in the \texttt{funm} MATLAB function.
  An analogous row-wise formula, involving the univariate functions $f_{\lambda_j^A}(y):=f(\lambda_j^A,y)$, applies to the case where only $A$ is diagonalized.  The resulting procedures are denoted by \textsc{fun2\_diagA} and \textsc{fun2\_diagB} and are reported in Algorithm~\ref{alg:fun2_diagA} and Algorithm~\ref{alg:fun2_diagB}.
  
  		\begin{minipage}[t]{.45\linewidth}
  	\begin{algorithm}[H] 
  		\small 
  		\caption{}\label{alg:fun2_diagA}
  		\begin{algorithmic}[1]
  			\Procedure{fun2\_diagA}{$f,A,B,C$}
  			\State $[S_A, D_A]=\texttt{eig}(A)$
  			\State $\widetilde C\gets S_A^{-1}C$
  			\State $D\gets \mathbf 0_{m\times n}$
  			\For{$j=1,\dots m$}
  			\State $D(j,\ :)\gets \widetilde C(j,\ :)f_{\lambda_j^A}(B)$
  			\EndFor
  			\State \Return $S_AD$
  			\EndProcedure
  		\end{algorithmic}
  	\end{algorithm} 
  \end{minipage}~
  \begin{minipage}[t]{.45\linewidth}
  	\begin{algorithm}[H] 
  		\small 
  		\caption{}\label{alg:fun2_diagB}
  		\begin{algorithmic}[1]
  		\Procedure{fun2\_diagB}{$f,A,B,C$}
  		  		\State $[S_B, D_B]=\texttt{eig}(B)$
  		  		  			\State $\widetilde C\gets CS_B^{-T}$
  			\State $D\gets \mathbf 0_{m\times n}$
  			\For{$j=1,\dots n$}
  			\State $D( :, \ j)\gets f_{\lambda_j^B}(A)\widetilde C( :, \ j)$
  			\EndFor
  			\State \Return $DS_B^T$
  			\EndProcedure
  		\end{algorithmic}
  	\end{algorithm} 
  \end{minipage} 
\subsection{Contribution}
From now on we will consider $f\{A,B^T\}(C)$ (instead of $f\{A,B\}(C)$) because this simplifies the exposition.
We propose a numerically reliable method for computing $f\{A,B^T\}(C)$ for a general function $f(x,y)$
  without requiring that $A$ and/or $B$ can be diagonalized with a well conditioned similarity transformation.
 In complete analogy to the  univariate scenario, 
our procedure computes the Schur decompositions  $A=Q_AT_AQ_A^*$ and $B=Q_BT_BQ_B^*$, so that
the task boils down to evaluate the bivariate function for triangular coefficients:
\[
f\{A,B^T\}(C)= Q_Af\{T_A,T_B^T\}(\widetilde C)Q_B^*,\qquad \widetilde C:=Q_A^*CQ_B.
\]
A generalized block recurrence is applied to retrieve 
$f\{T_A,T_B^T\}(\widetilde C)$; the recursion requires
to compute $f$ on pairs of diagonal blocks of $T_A$ and 
$T_B^T$ and to solve Sylvester equations involving either 
diagonal blocks of $T_A$ or of $T_B$. In view of the latter
operation, we need to reorder the Schur forms of $A$ and $B$
such that distinct diagonal blocks have sufficiently separated
eigenvalues. Finally, we evaluate $f$ on the smallest diagonal
blocks of $T_A$ and $T_B^T$, the so-called \emph{atomic blocks},
via a truncated bivariate Taylor expansion or, in the spirit
of \cite{higham2020multiprecision}, with a randomized 
approximate diagonalization technique combined with high 
precision arithmetic. As we discuss in 
Section~\ref{sec:block-diagonalization}, the procedure can be 
interpreted as an implicit (recursive) block-diagonalization strategy, 
where the eigenvectors matrices are not formed explicitly. 

The paper is organized as follows; in Section~\ref{sec:main} 
we describe the various steps of the algorithm in detail. 
In particular, Section~\ref{sec:blocking} discusses the blocking
procedure, Section~\ref{sec:atom} contains the two 
implementations of the function evaluation of the atomic 
blocks and Section~\ref{sec:splitting} provides further 
information about implementation aspects and complexity 
analysis. The focus of Section~\ref{sec:bartels} is on the connection of our method
with block diagonalization and the Bartels--Stewart 
algorithm. Numerical results are reported in 
Section~\ref{sec:numerical} and  conclusions are 
drawn in Section~\ref{sec:conclusion}.

\section{Recursive block diagonalization for bivariate matrix
functions}\label{sec:main}
The univariate Schur-Parlett algorithm computes $f(A)$, 
for a triangular $A$, by exploiting that $A$ and $f(A)$ 
commute. This property leads to a set of equations that allows
to retrieve the entries of $f(A)$ a
superdiagonal at a time, starting with the diagonal elements.

A natural question is whether the triangular structure 
of $A$ and $B$ can be exploited in the bivariate framework. 
However, here the situation is a bit different because 
the goal is to compute the application of $f\{A,B^T\}$ 
to  a matrix argument; the correspondent univariate operation is computing 
$f(A)v$ for a given vector $v$, for which the
Schur-Parlett scheme is not applicable.
Our strategy leverages the triangular structure 
of $A$ and $B$ to split the computation into smaller 
tasks. In order to show how the splitting works we state the
following technical result.

\begin{lemma}\label{lem:tech}
	Let $A\in\mathbb C^{m\times m}$ be a triangular matrix block partitioned as
	$$
	A=\begin{bmatrix}
	A_{11}& A_{12}\\
	& A_{22}
	\end{bmatrix}
	,
	$$
	where $A_{11}$ and $A_{22}$ are square matrices with no eigenvalue in common.
	Then, $\forall z\in\mathbb C\setminus \Lambda_A$:
	$$
	(zI-A)^{-1}=\begin{bmatrix}
	(zI-A_{11})^{-1}& (zI-A_{11})^{-1}V-V(zI-A_{22})^{-1}\\
	& (zI-A_{22})^{-1}
	\end{bmatrix}
	$$
	where $V$ solves the Sylvester equation
	$A_{11}V-VA_{22}=A_{12}$.
\end{lemma}
\begin{proof}
	Applying the block inverse formula we get 
	$$
	(zI-A)^{-1}=\begin{bmatrix}
	(zI-A_{11})^{-1}& (zI-A_{11})^{-1}A_{12}(zI-A_{22})^{-1}\\
	& (zI-A_{22})^{-1}
	\end{bmatrix}.
	$$
	Then, by imposing $(zI-A_{11})^{-1}V-V(zI-A_{22})^{-1}= (zI-A_{11})^{-1}A_{12}(zI-A_{22})^{-1}$ we get
	\begin{align*}
	A_{12}= V(zI-A_{22})-(zI-A_{11})V
	\quad \iff \quad 
	A_{12}=A_{11}V-VA_{22}.
	\end{align*}
\end{proof}
We are now ready to state the result that is at the core of our recursion for evaluating bivariate matrix functions.
\begin{theorem}\label{thm:splitting}
	Let $A\in\mathbb C^{m\times m}$ and $B\in\mathbb C^{n\times n}$ be  triangular matrices block partitioned as
	$$
	A=\begin{bmatrix}
	A_{11}& A_{12}\\
	& A_{22}
	\end{bmatrix}
	,\qquad B=\begin{bmatrix}
	B_{11}& B_{12}\\
	& B_{22}
	\end{bmatrix}, $$
	where $A_{11}\in\mathbb C^{(m-k_A)\times (m-k_A)}$ and  $A_{22}\in\mathbb C^{k_A\times k_A}$ have no eigenvalue in common and the same holds for $B_{11}\in\mathbb C^{(n-k_B)\times (n-k_B)}$ and  $B_{22}\in\mathbb C^{k_B\times k_B}$. 
	If $f(x,y)$ is a bivariate function analytic on $\Lambda_A\times\Lambda_B$ and $C=\begin{bmatrix}
	C_{11}&C_{12}\\ C_{21}&C_{22}
	\end{bmatrix}\in\mathbb C^{m\times n}$ is partitioned accordingly to $A$ and $B$, then we have
	\begin{align*}
	f\{A,B^T\}(C)&=\begin{bmatrix}
	I \\ 0
	\end{bmatrix}
	f\{A_{11},B_{11}^T\}\left(C_{11}+V C_{21}\right)\begin{bmatrix}
	I&W
	\end{bmatrix}\\
	&+\begin{bmatrix}
	-V\\ I
	\end{bmatrix}f\{A_{22},B_{11}^T\}\left(C_{21}\right)\begin{bmatrix}
	I& W
	\end{bmatrix}\\
	&+\begin{bmatrix}
	I \\ 0
	\end{bmatrix}f\{A_{11},B_{22}^T\}\left(\begin{bmatrix}
	I& V
	\end{bmatrix}C \begin{bmatrix}
	-W\\ I
	\end{bmatrix}\right) \begin{bmatrix}
	0 & I
	\end{bmatrix}\\
	&+\begin{bmatrix}
	-V\\ I
	\end{bmatrix} f\{A_{22},B_{22}^T\}\left(C_{22}-C_{21}W\right)
	\begin{bmatrix}
	0& I
	\end{bmatrix},
	\end{align*}
	where $V$ satisfies $A_{11}V-VA_{22}=A_{12}$ and  $W$ satisfies $B_{11}W-WB_{22}=B_{12}$.
\end{theorem}
\begin{proof}
	Let us indicate with $\mathfrak L_{j}(x):=(xI-A_{jj})^{-1}$ and $\mathfrak \mathfrak \mathfrak R_{j}(y):=(yI-B_{jj})^{-1}$, $j=1,2$, the resolvent functions associated with the diagonal blocks. By applying Lemma~\ref{lem:tech}  we get
		\begin{align*}
		f\{A, B^T\}(C) &=		\oint_{\Gamma_A}\oint_{\Gamma_B} \resizebox{.7\textwidth}{!}{$f(x,y)\begin{bmatrix}
		\mathfrak L_{1}(x)& \mathfrak L_{1}(x)V-V\mathfrak L_{2}(x)\\
		& \mathfrak L_{2}(x)
		\end{bmatrix} C \begin{bmatrix}
		\mathfrak R_{1}(y)& \mathfrak R_{1}(y)W-W\mathfrak R_{2}(y)\\
		& \mathfrak R_{2}(y)
		\end{bmatrix} \ dx dy$}\\
	&=\begin{bmatrix}
	I \\ 0
	\end{bmatrix}\oint_{\Gamma_A} \oint_{\Gamma_B} f(x,y)\mathfrak L_{1}(x)\begin{bmatrix}
	C_{11}+VC_{21}
	\end{bmatrix}\mathfrak R_{1}(y)dx\, dy\begin{bmatrix}
	I& W
	\end{bmatrix}\\
	&+\begin{bmatrix}
	-V\\ I
	\end{bmatrix}\oint_{\Gamma_A}\oint_{\Gamma_B}f(x,y)\mathfrak L_{2}(x)C_{21}\mathfrak R_{1}(y)dx\, dy\begin{bmatrix}
	I& W
	\end{bmatrix}\\
	&+\begin{bmatrix}
	I \\ 0
	\end{bmatrix}\oint_{\Gamma_A}\oint_{\Gamma_B}f(x,y)\mathfrak L_{1}(x)\begin{bmatrix}
	I& V
	\end{bmatrix}C \begin{bmatrix}
	-W\\ I
	\end{bmatrix}\mathfrak R_{2}(y)dx\, dy \begin{bmatrix}
	0 & I
	\end{bmatrix}\\
	&+\begin{bmatrix}
	-V\\ I
	\end{bmatrix}\oint_{\Gamma_A}\oint_{\Gamma_B}f(x,y)\mathfrak L_{2}(x)\left(
	C_{22}-C_{21}W\right)
	\mathfrak R_{2}(y)dx\, dy\begin{bmatrix}
	0& I
	\end{bmatrix}.
	\end{align*}
\end{proof}
In the $2\times 2$ case we can leverage the previous 
result to state a generalization for the formula of 
univariate functions of $2 \times 2$ upper triangular 
matrices using divided differences. In the univariate
case, we have
\cite[Theorem 4.11]{higham2008functions}
\[
f\left(\begin{bmatrix}
\lambda_1& a_{12}\\
&\lambda_2
\end{bmatrix}\right)=\left(\begin{bmatrix}
f(\lambda_1)& a_{12}D_x[\lambda_1, \lambda_2]f\\
&f(\lambda_2)
\end{bmatrix}\right)\]
where $D_x$ denotes the one dimensional divided difference 
\[D_x[\lambda_1, \lambda_2]f=\begin{cases}
\frac{f(\lambda_2)-f(\lambda_1)}{\lambda_2-\lambda_1}&\lambda_1\neq \lambda_1\\
f'(a_{11})&\lambda_1=\lambda_2
\end{cases}.
\]
We use the following definition of divided differences for bivariate functions:
\[
D_x[x_1,x_2]f(x,y):=\frac{f(x_2, y)-f(x_1,y)}{x_2-x_1},\qquad D_y[y_1,y_2]f(x,y):=\frac{f(x,y_2)-f(x, y_1)}{y_2-y_1}.
\]
Note that $D_x[x_1,x_2]f(x,y)$ is a univariate function of $y$.
Applying Theorem~\ref{thm:splitting}  yields the following formula that express $f\{A,B^T\}(C)$ in terms of $f(x,y)$ and its divided differences evaluated at all the possible  pairs of eigenvalues of $A$ and $B$.

\begin{corollary} \label{cor:2x2}
	Let 
	\[
	A=\begin{bmatrix}
	\lambda_1&a_{12}\\ &\lambda_2
	\end{bmatrix},\quad B=\begin{bmatrix}
	\mu_1&b_{12}\\ &\mu_2
	\end{bmatrix}
	,\quad C=\begin{bmatrix}
	c_{11}&c_{12}\\ c_{21}&c_{22}
	\end{bmatrix}
	\]
	and $f(x,y)$ such that $f\{A,B^T\}(C)$ is well defined. Then
	\begin{align*}
	f\{A,B^T\}(C)&= 
	\begin{bmatrix}
	  f(\lambda_1, \mu_1) & f(\lambda_1, \mu_2) \\
	  f(\lambda_2, \mu_1) & f(\lambda_2, \mu_2) \\
	\end{bmatrix} \circ C \\
	&+\begin{bmatrix}
	c_{21}a_{12}D_x[\lambda_1, \lambda_2]f(x,\mu_1) & \Delta\\
	& c_{21}b_{12}D_y[\mu_1,\mu_2]f(\lambda_2,y)
	\end{bmatrix}
	\end{align*}	
	where $\circ$ denotes the Hadamard product, 
	and \begin{align*}\Delta&:= c_{22}a_{12}D_x[\lambda_1, \lambda_2]f(x,\mu_2)+c_{11}b_{12}D_y[\mu_1,\mu_2]f(\lambda_1,y)\\ &+c_{21}a_{12}b_{12}D_{x}[\lambda_1,\lambda_2]D_{y}[\mu_1,\mu_2]f(x,y).\end{align*}
	\end{corollary}

Going back to the general framework, 
Theorem~\ref{thm:splitting} splits the computation of 
$f\{A,B^T\}(C)$ into $4$ bivariate functions of triangular
coefficients with smaller sizes, the solution of $2$ 
Sylvester equations and some matrix multiplications and 
additions. Applying this procedure recursively, reduces the 
problem to evaluate bivariate matrix functions on scalars 
or $2\times 2$ triangular matrices via the formula 
in Corollary~\ref{cor:2x2}. In practice, it is convenient 
to stop the recursion at a larger block size in order to 
exploit BLAS3 operations.
We note that the Sylvester equations solved in the 
four branches generated by a recursion are pairwise 
identical, since they only depend on $A$ or $B$. Therefore, 
the most efficient implementation solves these equations 
before the recursive call. For readability, this is not 
done in Algorithm~\ref{alg:fun2m} but we  discussed 
this step in further detail in Section~\ref{sec:splitting}.  

The implementation of this approach requires
the availability of two additional procedures, in the
 spirit of the univariate Schur-Parlett algorithm 
 \cite{davies2003schur}:
\begin{description}
	\item[\sc fun2\_atom]  evaluates the function for sufficiently small matrix arguments,
	\item[\sc blocking] produces the blocking pattern that ensures a sufficient separation between the spectra of the diagonal blocks; it returns the ordering permutation and the list of indices for each block
	$\mathcal I^A$ and $\mathcal I^B$, respectively.
	\end{description}
In particular, \textsc{fun2\_atom} aims at computing 
$f\{A,B^T\}(C)$ for input arguments of size up to 
$n_{\min}\times n_{\min}$, where the choice of 
$n_{\min}$ depends on the conditioning of the problem 
or on the underlying computer architecture. 

The recursion is constructed by repeatedly splitting the index
partitionings $\mathcal I^A = \mathcal I^{A_1} \sqcup \mathcal I^{A_2}$ 
and $\mathcal I^B = \mathcal I^{B_1} \sqcup \mathcal I^{B_2}$ in two parts, and 
applying Theorem~\ref{thm:splitting} with $A_{ii} = A(\mathcal I^{A_i}, \mathcal I^{A_i})$ and $B_{jj} = B(\mathcal I^{B_j}, \mathcal I^{B_j})$.
The purpose of \textsc{blocking} is to ensure that the spectra of $A_{11}$ and $A_{22}$ (resp. $B_{11}$ and $B_{22}$) are sufficiently separated
at all steps of recursion; this is a necessary condition for solving accurately the Sylvester equations encountered in the process.

The detailed descriptions of \textsc{fun2m\_atom} 
and \textsc{blocking} is postponed to the next sections. 
The algorithm obtained using this paradigm is reported 
in Algorithm~\ref{alg:fun2m}. The pseudocode makes also
use of the function \textsc{Sylvester\_tri} which solves
Sylvester matrix equations with triangular coefficients; 
this can be done very efficiently, as described in 
\cite{jonsson2002recursive}; in our code, we simply rely on the 
triangular sylvester solver included in the LAPACK routine 
\texttt{*trsyl}. 

\begin{remark}
	We note that in Theorem~\ref{thm:splitting} it is possible to only partition $A$ or $B$, 
	instead of both matrices at once. This splits the problem into two subtasks. Formally, this 
	operation can be seen as a particular case of Theorem~\ref{thm:splitting} where either 
	$A_{22}$ or $B_{22}$ are empty matrices, and the associated terms in the expression of 
	$f\{A, B^T\}(C)$ disappear. 
\end{remark}

  \begin{algorithm} 
	\small 
	\caption{Evaluates $f\{A,B^T\}(C)$}\label{alg:fun2m}
	\begin{algorithmic}[1]
		\Procedure{fun2m}{$f,A,B,C$}
\If{$A$ and $B$ are normal}
\Return \Call{fun2\_diag}{$f,A,B,C$}
\ElsIf{$A$ is normal}
\Return \Call{fun2\_diagA}{$f,A,B,C$}
\ElsIf{$B$ is normal}
\Return \Call{fun2\_diagB}{$f,A,B,C$}
\Else
\State $[Q_A,T_A]=\texttt{schur}(A)$
\State $[Q_B,T_B]=\texttt{schur}(B)$
\State $[P_A, \mathcal I^A]=\Call{blocking}{T_A}$
\State $[P_B, \mathcal I^B]=\Call{blocking}{T_B}$
\State $T_A\gets P_A^*T_AP_A$, $T_B\gets P_B^*T_BP_B$
\State $\widetilde C\gets P_A^*Q_A^*CQ_BP_B$
\State$F\gets$\Call{fun2m\_rec}{$f,T_A,T_B,\widetilde C$, $\mathcal I^A$,$\mathcal I^B$}
\State \Return $Q_AP_AFP_B^*Q_B^*$
		\EndIf
		\EndProcedure
	\end{algorithmic}
\begin{algorithmic}[1]
\Procedure{fun2m\_rec}{$f,A,B,C,\mathcal I^A,\mathcal I^B$}
\State $\ell_A\gets \texttt{length}(\mathcal I^A)$\Comment{$\mathcal I^A=\{I_1^A,\dots, I_{\ell_A}^A\}$}
\State $\ell_B\gets \texttt{length}(\mathcal I^B)$\Comment{$\mathcal I^B=\{I_1^B,\dots, I_{\ell_B}^B\}$}
\If{$\ell_A$ or $\ell_B$ is zero}
\Return $[\ ]$
\ElsIf{$\ell_A$ and $\ell_B$ are both equal to $1$}
\Return \Call{fun2\_atom}{$f,A,B,C$}
\Else
\State Split $\mathcal I^{A}=\mathcal I^{A_1}\sqcup \mathcal I^{A_2}$ and $\mathcal I^{B}=\mathcal I^{B_1}\sqcup \mathcal I^{B_2}$ \Comment{see Section~\ref{sec:splitting}}
\State Partition $A,B$ and $C$ according to $\mathcal I^{A_1},\mathcal I^{A_2},\mathcal I^{B_1},\mathcal I^{B_2}$:
\[
A=\begin{bmatrix}
A_{11}&A_{12}\\ &A_{22}
\end{bmatrix},\qquad B=\begin{bmatrix}
B_{11}&B_{12}\\ &B_{22}
\end{bmatrix},\qquad C=\begin{bmatrix}
C_{11}&C_{12}\\ C_{21}&C_{22}
\end{bmatrix}
\]
\State $V\gets \Call{Sylvester\_tri}{A_{11},A_{22},A_{12}}$\Comment{$V,W$ are precomputed, see Section~\ref{sec:splitting}}
\State $W\gets \Call{Sylvester\_tri}{B_{11},B_{22},B_{12}}$
\State $C_1\gets C_{11}+VC_{21}$, $C_2\gets C_{21}$ \State $C_3\gets C_{12}-C_{11}W-VC_{21}W+VC_{22}$, $C_4\gets C_{22}-C_{21}W$
\State $F_1\gets \Call{fun2m\_rec}{f, A_{11},B_{11}, C_1,\mathcal I^{A_1}, \mathcal I^{B_1}}$
\State $F_2\gets \Call{fun2m\_rec}{f, A_{22},B_{11}, C_2, \mathcal I^{A_2},\mathcal I^{B_1}}$
\State $F_3\gets \Call{fun2m\_rec}{f, A_{11},B_{22}, C_3,\mathcal I^{A_1}, \mathcal I^{B_2}}$
\State $F_4\gets \Call{fun2m\_rec}{f, A_{22},B_{22}, C_4, \mathcal I^{A_2},\mathcal I^{B_2}}$
\State \Return $\begin{bmatrix}
F_1-VF_2& F_1W- VF_2W+F_3-VF_4  \\ F_2& F_2W+F_4
\end{bmatrix}$
\EndIf
\EndProcedure
\end{algorithmic}
\end{algorithm}

\subsection{Block partitioning of the Schur forms}  \label{sec:blocking}

Algorithm~\ref{alg:fun2m} requires the solutions of two Sylvester equations 
at every recursive step. In order to avoid an excessive error propagation, 
we need to ensure a sufficient spectral separation between the 
coefficients $A_{11}, A_{22}$, or $B_{11}, B_{22}$. This is in complete
analogy with the univariate Schur-Parlett algorithm, where only 
one matrix is involved. Hence, we rely on the same blocking procedure 
proposed in \cite[Algorithm~4.1]{davies2003schur} that, chosen a parameter $\delta > 0$, returns two index partitionings 
\[
    \mathcal I^A = \{ I^A_1, \ldots, I_{\ell_A}^A \}, \qquad 
    \mathcal I^B = \{ I^B_1, \ldots, I_{\ell_B}^B \}, 
\]
where $I^A_i \subseteq \{ 1, \ldots, m \}$ and 
$I^B_j \subseteq \{ 1, \ldots, n \}$, that 
identify diagonal blocks $A(I_i^A, I_i^A)$ and 
$B(I_j^B, I_j^B)$ with the following properties:
\begin{itemize}
    \item Any block of size at least $2 \times 2$
      is such that for 
      each eigenvalue $\lambda$ there exists another eigenvalue
      $\mu$ in the same block satisfying $|\lambda - \mu| \leq \delta$.
    \item Each pair of eigenvalues $\lambda, \mu$ that belong to different blocks in the same matrix ($A$ or $B$)
      have distance at least $|\lambda - \mu| > \delta$. 
\end{itemize}
The first property is useful to construct a polynomial approximant that 
is accurate on the spectrum of the block; for instance, a truncated Taylor
expansion. We will use this fact in Section~\ref{sec:taylor}, while this 
will not be relevant for the perturb-and-diagonalize approach in Section~\ref{sec:perturbanddiag}. 

In practice, \textsc{blocking} interprets the eigenvalues as nodes 
in a graph, which are connected by an edge if their distance is less than
$\delta$; then, the blocking corresponds to identifying the connected
components of this graph, and to reorder the Schur form accordingly.

We remark that the condition $|\lambda - \mu| > \delta$ does not 
guarantee that the Sylvester equations solved in the recursion 
are well-conditioned, since their coefficients are non-normal. Hence, 
we propose to verify this a posteriori, and possibly cure the ill-conditioning  by merging 
the blocks. This approach is described in detail in Section~\ref{sec:splitting}; however,
it might not be applicable when employing Taylor expansions for 
evaluating the function at the atomic blocks, due to the potential loss of 
spectral clustering. 

\subsection{Evaluating the function at the atomic blocks}\label{sec:atom}
In this section we specify two implementations of \textsc{fun2\_atom}; the first is based on the evaluation of a truncated (bivariate) Taylor expansion
and requires the availability of the partial derivatives of arbitrary
orders; the second relies on the recent perturb-and-diagonalize approach developed in \cite{higham2020multiprecision} which is
derivative-free. In addition to the spectra separation
for different blocks, the Taylor approach requires the 
blocking strategy to provide matrices with
sufficiently clustered eigenvalues. For the 
perturb-and-diagonalize approach this is not necessary and
we choose  the block size to be of the order of $n_{\min}=4$ 
if this can be achieved along with the spectra separation 
condition. 

Throughout this section, $\norm{\cdot}$ denotes the spectral norm.
\subsubsection{Bivariate Taylor expansion} \label{sec:taylor}
Let us assume that $A$ and $B$ are triangular matrices with eigenvalues clustered around $\lambda = \trace(A) / m$ 
and $\mu = \trace(B)/n$, respectively; that is  $\Lambda_A\subset \mathcal B(\lambda, r_A):=\{|z-\lambda|< r_A\}$ and $\Lambda_B\subset \mathcal B(\mu, r_B):=\{|z-\mu|< r_B\}$ for $r_A,r_B>0$.

We consider a
truncated Taylor expansion of $f(x,y)$ centered 
at $(\lambda, \mu)$:
\[
  f(x,y) = \sum_{i+j \leq k}
    \frac{f^{(i,j)}(\lambda, \mu)}{i!j!} 
    (x - \lambda)^i 
    (y - \mu)^j + R_{k}(x, y)
\]
which leads to the following approximation (see \cite[Section 2.2]{kressner2011bivariate}), 
\begin{equation}\label{eq:taylor-approx}
  f\{A,B^T\}(C) \approx \sum_{i+j\leq k} 
    \frac{f^{(i,j)}(\lambda, \mu)}{i!j!}
    N_A^i C N_B^j, 
\end{equation}
where $A = \lambda I + N_A$ and $B = \mu I + N_B$. The value of 
$k$ is chosen to ensure a small error in the approximation:
\begin{equation} \label{eq:remainder}
  \left \lVert
    f\{A,B^T\}(C) - \sum_{i+j\leq k} \frac{f^{(i,j)}(\lambda, \mu)}{i!j!}N_A^i C N_B^j
  \right \rVert = \norm{R_{k}\{ A, B^T \}(C)}.
\end{equation}
The remainder in the above formula can be estimated using a straightforward
generalization of \cite[Theorem~2.5]{davies2003schur} to the 
bivariate case.
\begin{lemma}
  The remainder of the approximation in \eqref{eq:remainder}
  is bounded by:
  \[
    \norm{R_{k}\{ A, B^T \}(C)} \leq 
      \max \{ \norm{N_A},  \norm{N_B} \}^{k+1}
      \cdot \norm{C} \cdot 
      \max_{(\xi, \eta) \in \mathfrak B}
        \sum_{i+j=k+1} \frac{|f^{(i,j)}(\xi, \eta)|}{i!j!},
  \]
  where $\mathfrak B = \mathcal B(\lambda, r_x)
  \times \mathcal B(\mu, r_y)$. 
\end{lemma}

\begin{proof}
    We write the remainder in Lagrange form as follows:
    \begin{align*}
      R_k(x,y) &= \sum_{i+j \ge k+1} \frac{f^{(i,j)}(\lambda, \mu)}{i!j!}
      (x - \lambda)^i (y - \mu)^j \\
      &= \sum_{i+j=k+1} \frac{f^{(i,j)}(\xi, \eta)}{i!j!}
      (x - \lambda)^i (y - \mu)^j, 
    \end{align*}
    where $(\xi, \eta)$ belong to the segment that
    connects $(\lambda, \mu)$ with $(x,y)$. Evaluating $R_k(x,y)$
    at $A$ and $B$ applied to $C$ yields
    \begin{align*}
      \norm{R_k\{A,B^T\}(C)} &\leq 
      \sum_{i+j=k+1} \frac{f^{(i,j)}(\xi, \eta)}{i!j!}
      \max \{ \norm{N_A}, \norm{N_B} \}^{k+1} \norm{C} \\
      &\leq \max \{ \norm{N_A},  \norm{N_B} \}^{k+1}
      \cdot \norm{C} \cdot 
      \max_{(\xi, \eta) \in \mathfrak B}
        \sum_{i+j=k+1} \frac{|f^{(i,j)}(\xi, \eta)|}{i!j!}.
    \end{align*}

\end{proof}
In order to compute an approximation of the form \eqref{eq:taylor-approx} that yields an accuracy $\epsilon$, we propose the following scheme:
\begin{enumerate}
    \item Compute $\theta :=\max\{\norm{N_A},\norm{N_B}\}$, and define $\mathfrak B_{m,n}:=\Lambda_A\times \Lambda_B$
    \item Identify the minimal integer $k$ such that
    $$
    \theta^{k+1}
      \cdot \norm{C} \cdot 
        \sum_{i+j=k+1} \frac{|f^{(i,j)}(\lambda, \mu)|}{i!j!}\leq \epsilon,
    $$
    \item Verify that the chosen $k$ satisfies 
    also the following inequality:
    $$
    \theta^{k+1}
      \cdot \norm{C} \cdot 
      \max_{(\xi, \eta) \in \mathfrak B_{m,n}}
        \sum_{i+j=k+1} \frac{|f^{(i,j)}(\xi, \eta)|}{i!j!}\leq \epsilon,
    $$
    If not, increase $k$ checking again the previous conditions,
    until both are satisfied.
    \item Using the computed derivatives, evaluate \eqref{eq:taylor-approx}.
\end{enumerate}
Note that the replacing $\mathfrak B$ with $\mathfrak B_{m,n}$ does not guarantee the upper bound for the remainder of the Taylor expansion, although it is in general a good heuristic.  
The  procedure sketched above is reported in 
Algorithm~\ref{alg:fun2-atom-taylor}.
\begin{algorithm}[H] 
  	\small 
  	\caption{Computes $f\{A,B^T\}(C)$ for triangular $A,B$ with a Taylor expansion}\label{alg:fun2-atom-taylor}
  	\begin{algorithmic}[1]
  		\Procedure{fun2\_atom\_taylor}{$f,A,B,C, \epsilon$}
  		\State Retrieve $\lambda$ and $\mu$ from the diagonals of $A$ and $B$
  		\State $N_A\gets A-\lambda I$,  $N_B\gets B-\mu I$
  		\State $\theta\gets \max\{\norm{N_A},\norm{N_B}\}$
  		\For{$k=1,\dots,k_{\max}$}
  		\State $R\gets  \theta^{k+1}
      \cdot \norm{C} \cdot 
        \sum_{i+j=k+1} \frac{|f^{(i,j)}(\lambda, \mu)|}{i!j!}$
        \If{$R\leq \epsilon$}
            \State $R_2\gets  \theta^{k+1}
      \cdot \max\limits_{(\xi, \eta) \in \mathfrak B_{m,n}}
        \sum_{i+j=k+1} \frac{|f^{(i,j)}(\xi, \eta)|}{i!j!}$
        \If{$R_2\leq \epsilon$}
        \State \textbf{break}
        \EndIf
        \EndIf
  		\EndFor
  		\State \Return $\sum_{i+j\leq k} 
    \frac{f^{(i,j)}(\lambda, \mu)}{i!j!}
    N_A^i C N_B^j$ \label{step:biv-poly}
  		\EndProcedure
  	\end{algorithmic}
  \end{algorithm} 

To conclude, we specify how the bivariate polynomial at line~\ref{step:biv-poly} is evaluated. Given any polynomial
$P(x,y)$ of total degree $k$ we write it as follows:
\[
  P\{A, B^T\}(C) = \sum_{i+j \leq k} p_{ij} A^i  C B^j = 
  \sum_{i = 0}^k A^i C 
  \underbrace{\sum_{j = 0}^{k - i} p_{ij} B^j}_{P_i(B)}. 
\]
Then, we evaluate $P_i(B)$ for $i = 0, \ldots, k$ using the
Horner scheme, and finally we compute $\sum_{i=0}^k A^i C P_i(B)$
using again the Horner scheme with respect to the variable $A$:
\[
  P\{A, B^T\}(C) = A ( 
    \ldots A(ACP_{k}(B) + C P_{k-1}(B)) + C P_{k-2}(B) + \ldots
  ) + C P_0(B).
\]
This approach requires $k(k-1)/2$ multiplications between 
$n \times n$ matrices, $k$ multiplications between $m \times n$
and $n \times n$ matrices, and finally $k$ multiplications between
$m \times m$ and $m \times n$ matrices. This yields the total cost
of $\mathcal O(k^2 n^3 + k m^2n + k mn^2)$. If $n > m$, it is 
convenient to swap the role of $A$ and $B$, relying on 
an analogous formula. 

%

\subsubsection{Perturb-and-diagonalize} \label{sec:perturbanddiag}

A derivative-free approach for the evaluation of
$f(A)$, when $A$ is highly non-normal, has been proposed
in \cite{davies2008approximate}. The idea is to introduce a small random 
perturbation $E$ to the matrix $A$, so that
$A+E$ is diagonalizable with probability $1$. Then,
$f(A + E) \approx f(A)$ is evaluated by diagonalization. 
The method has been
recently improved in \cite{higham2020multiprecision},
and has been proposed for evaluating the atomic blocks
in the Schur-Parlett scheme. In particular, in \cite{higham2020multiprecision}
it is suggested to first compute the Schur form, introduce a diagonal perturbation and evaluate the function of the perturbed Schur form using a higher precision determined by estimating the 
condition number of its eigenvector matrix.

We propose to rely on the analogue scheme 
in the bivariate case.
More specifically, consider $A, B$ upper triangular
matrices, and $E_A, E_B$ small diagonal
perturbations. Let
\[
  \tilde A := A + E_A = V_A D_A V_A^{-1}, \qquad 
  \tilde B := B + E_B = V_B D_B V_B^{-1}
\]
be the the eigendecompositions of the perturbed
matrices. Thanks to the triangular
structure of $A + E_A$ and $B + E_B$ the eigenvalues 
can be read off the diagonal, so that
$D_A$ and $D_B$ can be considered as not
affected by rounding errors. 
The eigenvector matrices $V_A, V_B$ are also triangular and are determined by solving
triangular shifted linear systems with $\tilde A, \tilde B$. As 
noted in \cite{higham2020multiprecision} this allows to estimate
$\kappa(V_A)$ and $\kappa(V_B)$ from the entries of $\tilde A, \tilde B$
using 
\begin{equation} \label{eq:kestimate}
    \kappa(V_A) \lesssim
      m\zeta(\zeta +1)^{m-2}, \qquad \zeta:= \frac{\max_{i<j} |\tilde A_{ij}|}{\min_{i \neq j} |\tilde A_{ii} - \tilde A_{jj}|}, 
\end{equation}
and analogously for $\kappa(V_B)$. We remark that Equation~\ref{eq:kestimate} can be pessimistic for moderate of values of $m$. As in \cite{higham2020multiprecision} we apply the following heuristic.
\begin{itemize}
\item[$(i)$] We further partition $A$ with \textsc{blocking} using $\delta_1<\delta$; in our experiments we adopt $\delta_1=5\cdot 10^{-3}$.
\item[$(ii)$] We estimate $\kappa(V_A)$ by the maximum of the quantities as in Equation~\ref{eq:kestimate} computed for its diagonal blocks.
\end{itemize}
In practice, the latter heuristic might fail for highly non normal matrices, therefore we verify it a posteriori as we describe later in this section.   

A classic result for univariate matrix functions bounds
the forward error of computing $f(A)$ by diagonalization
with a small constant multiplied by $\kappa(V) u$, where 
$u$ is the current unit roundoff, and 
$V$ is the eigenvector matrix of $A$ \cite[page 82]{higham2008functions}. 
We generalize the latter within the following result.

\begin{lemma} \label{lem:k1bound}
  Let $F = V_A f\{ D_A, D_B \} (V_A^{-1} C V_B) V_B^{-1}$, 
  with $D_A = \mathrm{diag}(\lambda_1, \ldots, \lambda_m)$,
  $D_B = \mathrm{diag}(\mu_1, \ldots, \mu_n)$, 
  and let $\hat F$ be the corresponding  quantity computed 
  in floating point arithmetic. 
  If the matrix multiplications
  are performed exactly, and $f(\lambda_i, \mu_j)$ 
  is computed with relative error bounded by $u_h$, 
  then
  \[
    \norm{F - \hat F} \leq
      \kappa(V_A) \kappa(V_B) \norm{C} 
      \max_{i,j} |f(\lambda_i, \mu_j)| u_h.
  \]
\end{lemma}

\begin{proof}
    Under the assumptions, $\hat F$ is equal to
    \[
      V_A \left[ (G + E) \circ (V_A^{-1} C V_B) \right] V_B^{-1}, 
      \qquad 
      G_{ij} := f(\lambda_i, \mu_j), 
    \]
    where $|E_{ij}| \leq \max_{i,j} |f(\lambda_i, \mu_j)| u_h$. Then, 
    \begin{align*}
      \norm{F - \hat F} &\leq 
        \norm{V_A} \norm{E \circ (V_A^{-1} C V_B)} \norm{V_B^{-1}} \\
        &\leq \norm{V_A} 
        \max_{i,j} |f(\lambda_i, \mu_j)|
        \norm{V_A^{-1} C V_B} \norm{V_B^{-1}} u_h \\
        &\leq 
          \kappa(V_A) \kappa(V_B) \norm{C}
            \max_{i,j} |f(\lambda_i, \mu_j)|
        u_h.
    \end{align*}
\end{proof}

In view of Lemma~\ref{lem:k1bound} we 
choose $u_h$ to ensure that 
$\norm{F - \hat F} \leq \norm{F} u$, where $u$ is the current
machine roundoff. By assuming $\norm{F} \approx \max_{i,j} |f(\lambda_i, \mu_j)| \norm{C}$, similarly to what is done in \cite{higham2020multiprecision},
this can be achieved by setting:
\begin{equation} \label{eq:uh}
  u_h \leq \frac{\norm{F} u}{\kappa(V_A) \kappa(V_B) \norm{C}
            \max_{i,j} |f(\lambda_i, \mu_j)|}
    \approx \frac{u}{\kappa(V_A) \kappa(V_B)}
\end{equation}
In practice, the quantities $\kappa(V_A), \kappa(V_B)$ 
are estimated, before computing $V_A$ and $V_B$, using the right hand side of
Equation~\ref{eq:kestimate}. 
Then, $V_A$ and $V_B$ are computed
with a relative accuracy $u_h$ as in Equation~\ref{eq:uh} or better, as pointed out in the following remark. 

\begin{remark} \label{rem:matmul-wp}
  Assuming that the matrix multiplications are performed 
  exactly up to the current precision
  simplifies the analysis, and is required also in
  \cite{higham2020multiprecision}. This 
  is not particularly restrictive, since it
  can be guaranteed by computing matrix multiplications
  temporarily working with the lower unit roundoff 
  $u_h \cdot \max \{ \kappa(V_A), \kappa(V_B) \}^{-1}$. 
  The same argument applies when computing $V_A, V_B$
  by solving shifted linear systems with $A, B$. 
\end{remark}
Often, relying on $u_h$ as in Equation~\ref{eq:uh} 
where $\kappa(V_A)$ and $\kappa(V_B)$ are approximated 
via Equation~\ref{eq:kestimate}
yields a pessimistic estimate for the necessary working 
precision. 
Hence, 
once the triangular eigenvector matrices $V_A$ and $V_B$ 
are available, we propose to refine the estimates of their 
condition numbers and adjust the precision in the evaluation 
of the function of the atomic blocks. For efficiency reasons, 
we would like to avoid 
computing $\kappa(V_A),\kappa(V_B)$ with high precision, 
if possible. Hence we suggest this greedy strategy, which we 
describe for a generic $V\in\{V_A,V_B\}$: 
\begin{itemize}
	\item Convert $V$ to standard floating point precision and compute $\norm{V}$ and $\norm{V^{-1}}$; if $\norm{V}\cdot\norm{V^{-1}}\leq 10^{14}$ then return this value as a sufficiently accurate estimate for $\kappa(V)$.
	\item Otherwise construct the matrix $U$
	$$
	U_{ij}=\begin{cases}
	|V_{ij}|& i=j\\
	-|V_{ij}|&i\neq j
	\end{cases}
	$$
	for which the inverse can be computed entry-wise 
	in standard precision, and satisfies 
	$\norm{U^{-1}}\geq \norm{V^{-1}}$ \cite[Section 8.2]{higham2002accuracy}. If $\norm{U^{-1}}\leq 10^4 \norm{V^{-1}}$ (that is,
	the guaranteed estimate on the number of digits is not much more pessimistic than the previous one)
	then use  $\norm{V}\norm{U^{-1}}$ as upper bound for $\kappa(V)$.
	\item Finally, if none of the previous points succeeds, then compute $\kappa(V)$ 
	  using a unit roundoff $u_h$.  
\end{itemize}  

To sum up, we propose to evaluate $f\{A,B^T\}(C)$ with $A,B$ upper
triangular following these steps:
\begin{enumerate}
    \item Lower the unit roundoff to $u^2$, and perturb $A$ and $B$
      with diagonal matrices of norm $\norm{A} u$ and $\norm{B} u$, 
      respectively.
    \item Determine $u_h$ using Equation~\ref{eq:kestimate} and Equation~\ref{eq:uh}, and if 
      $u_h < u^2$, set the unit roundoff to $u_h$.
      \item Compute $V_A$ and $V_B$ using a unit roundoff $\frac{u_h}{\max\{\kappa(V_A),\kappa(V_B)\}}$, where $\kappa(V_A)$, $\kappa(V_B)$ are estimated with Equation~\ref{eq:kestimate}.
	  \item Refine the estimates for $\kappa(V_A)$ 
		  and $\kappa(V_B)$ with the  greedy strategy described
		  above and recompute the unit roundoff $u_h$. If the new estimates $\kappa(V_A),\kappa(V_B)$ are larger than the previous ones, we adjust the precision accordingly and we go back to 3., using these values instead of Equation~\ref{eq:kestimate}.  
    \item Run Algorithm~\ref{alg:fun2m_diag} to evaluate 
      $f\{A, B^T\}(C)$ using the new $u_h$. 
\end{enumerate}
The whole procedure is also summarized in Algorithm~\ref{alg:pertdiag}.

We remark that recomputing $u_h$ ensures a significant performance gain when 
all the blocks are of small size, as it allows 
to perform the calls to \textsc{fun2m\_diag} 
(which are $\mathcal O(mn)$) at a lower precision, at the price
of computing the condition numbers (which is only performed 
$\mathcal O(n + m)$ times).

\begin{algorithm} 
  	\small 
  	\caption{Computes $f\{A,B^T\}(C)$ for triangular $A,B$ with a 
  	perturb and diagonalize approach}\label{alg:pertdiag}
  	\begin{algorithmic}[1]
  		\Procedure{fun2\_atom\_diag}{$f,A,B,C$}
  		\State Set unit roundoff to $u^2$
  		\State Generate random diagonal matrices $E_A, E_B$ of norm
  		  $\norm{A}u, \norm{B}u$
  		\State $\tilde A \gets A + E_A, \tilde B \gets B + E_B$
  		\State Estimate $\kappa(V_A), \kappa(V_B)$ as in Equation~\ref{eq:kestimate}
  		\State $u_h \gets u / (\kappa(V_A) \kappa(V_B))$
  		\State Set the unit roundoff to $\min\{ u^2, u_h \}$
  		\State $[V_A, D_A] \gets \Call{Eig}{\tilde A}, \ \
  		  [V_B, D_B] \gets \Call{Eig}{\tilde B}$\Comment{Can be precomputed, see Section~\ref{sec:splitting}}
  		  \State Refine the estimates of $\kappa(V_A),\kappa(V_B)$ and set the unit roundoff to the new $u_h$
  		\State $F \gets \Call{fun2m\_diag}{D_A, D_B, V_A^{-1} C V_B}$
  		\State \Return $V_A F V_B^{-1}$
  		\EndProcedure
  	\end{algorithmic}
  \end{algorithm} 
  
\subsubsection{Avoiding complex arithmetic}
Whenever $A,B,C$ are real matrices and $f(x,y)$ has the 
property $\overline{f(x,y)}= f(\overline x,\overline y)$ (and 
in particular $f(x,y)$ is real for real arguments)
the bivariate matrix function $f\{A,B^T\}(C)$ is real as well. 
Indeed, we can select an integration path symmetric with respect
to the real axis in definition \eqref{eq:cauchy-def}, so that the
imaginary part of the integral is guaranteed to vanish. Hence, it 
is appealing to use an evaluation procedure that preserves 
the real structure. To this end, we first reduce the matrices 
$A$ and $B$ to real Schur form, so that the problem boils down
to dealing with
$2 \times 2$ blocks encoding complex conjugate eigenvalues. In 
fact, if the real structure is preserved by \textsc{fun2\_atom}, 
the recursion applied by \textsc{fun2m} only
requires solving Sylvester equations and matrix-matrix operations, 
that do not introduce any complex arithmetic. 

It is easy to see that the approach based on Taylor expansions preserves the real structure, if 
the complex conjugate eigenvalues have small imaginary parts, and thus can be put in the same block. 
Otherwise, for Taylor there is no straightforward alternative to working with complex arithmetic. 
In contrast, the perturb-and-diagonalize approach described in the previous section can be adapted to work directly with the real Schur form
without particular assumptions. 
The random perturbations of the diagonal blocks are chosen as $\left[\begin{smallmatrix}
	\delta\alpha&\delta\beta\\ -\delta\beta&\delta\alpha
	\end{smallmatrix}\right]$ in order to match the structure of the Schur form. 
Then, the latter is block diagonalized and 
$f\{ A, B^T\}(C)$ is evaluated, where $A,B$ are either $2 \times 2$ or
$1 \times 1$. When one between $A$ or $B$ is a $1 \times 1$ block, 
the problem can be recast into the evaluation of either 
a scalar function, or a univariate matrix
function
of a $2 \times 2$ block representing $z$ and $\overline{z}$; in the
latter case, 
the outcome can be expressed in terms of the block representing 
$g(z)$ and $\overline{g(z)}$, where $g(\cdot)$ is obtained by 
fixing the variable corresponding to the $1 \times 1$ block in $f(x,y)$. 

The following result provides an explicit formula for the 
case where both $A$ and $B$ are $2 \times 2$ blocks.

\begin{theorem}
	Let $A,B,C$ be $2\times 2$ real matrices with $A$ and $B$ of the form
	\[
	A=\begin{bmatrix}
	\alpha&\beta\\
	-\beta&\alpha
	\end{bmatrix},\qquad
		B=\begin{bmatrix}
	\gamma&\delta\\
	-\delta&\gamma
	\end{bmatrix}.
	\]
	If $f$ is such that $f\{A,B^T\}(C)$ is well defined and $\overline{f(x,y)}= f(\overline x,\overline y)$, then
	\[f\{A,B^T\}(C)=\frac12\begin{bmatrix} Q_1+Q_2& Q_3+Q_4\\
	Q_3-Q_4&Q_1-Q_2\end{bmatrix}\]
	where, denoting by $z=\alpha+\mathbf i \beta$ and $w=\gamma+\mathbf i\delta$, we have
	\begin{align*}
	Q_1&:= (c_{21}-c_{12})\Im(f(z,w)) + (c_{11}+c_{22})\Re(f(z,w)),\\
	Q_2&:= (c_{12}-c_{21})\Im(f(z,\overline w)) + (c_{11}-c_{22})\Re(f(z,\overline w)),\\
		Q_3&:= (c_{22}-c_{11})\Im(f(z,\overline w)) + (c_{12}+c_{21})\Re(f(z,\overline w)),\\
			Q_4&:= (c_{11}+c_{22})\Im(f(z,w)) + (c_{12}-c_{21})\Re(f(z,w)).	
	\end{align*}
\end{theorem} 
\begin{proof}
The matrices $A$ and $B$ are simultaneously diagonalized by means of the eigenvector matrix $\left[\begin{smallmatrix}
	1&1\\ \mathbf i&-\mathbf i
	\end{smallmatrix}\right]$. Then, applying formula \eqref{eq:diag-coef} and exploiting that $f(z,w)=\overline{f(\overline z,\overline w)}$ and $f(z,\overline w)=\overline{f(\overline z,w)}$ yields the claim.
	\end{proof}

We remark that in this case one needs to adjust the blocking 
procedure to make sure that conjugate pairs are kept together. In practice, we perform the blocking by only 
looking at real part of the eigenvalues, and then use the 
perturb-and-diagonalize algorithm for the evaluation at the
atomic blocks. 

\subsection{Splitting strategy and computational complexity}\label{sec:splitting}
We have not specified yet the splitting strategy for the
block index sets $\mathcal I^A$ and $\mathcal I^B$ returned
by the \textsc{blocking} procedure. Our code implements two
different possibilities (which we call \textbf{balanced} and 
\textbf{single}) that we detail at the end of this section. 
Under minimal assumptions, any splitting
strategy yields an algorithm with cubic cost in the sizes 
of $A$ and $B$. 
From now on, we make the following assumption. 

\begin{assumption}\label{ass:splitting}
The partitionings $\mathcal I^A$ and $\mathcal I^B$
are split in the same way in all branches of the recursion
of Algorithm~\ref{alg:fun2m}.
\end{assumption}  
Assumption~\ref{ass:splitting} implies that a given block in $A$ or $B$
is split in the same way in all branches of recursion. This allows to look at $A$ and $B$ separately and precompute the solutions of all Sylvester equations before running the recursion in Algorithm~\ref{alg:fun2m}. During this process,
we also check the conditioning of the equations; if ill-conditioning 
is detected, we adjust the blocking, as described in the next subsection.
Similarly, the eigendecompositions of the atomic blocks
are precomputed when using \textsc{fun2m\_atom\_diag} for the evaluations
of the atomic blocks. 

Note that, this strategy identifies two trees describing the recursive partitioning of the index sets of $A$ and $B$. We denote by $d_A$ and $d_B$ the depths of such trees.

\subsubsection{Dealing with ill-conditioned Sylvester equations}

The condition $|\lambda - \mu| > \delta$ obtained from 
the blocking strategy of Section~\ref{sec:blocking} does not
necessarily guarantee that the Sylvester equations
related with the diagonal blocks of 
$A$ and $B$ are well-conditioned, because their
coefficients are not normal.

Nevertheless, it is in general a good heuristic, and we 
propose to check a posteriori whether the condition number 
is larger than expected by verifying the norm of the solution.
More specifically, for a Sylvester equation $A_{11} V - VA_{22} = A_{12}$
we compute the ratio $r := \norm{V} / \norm{A_{12}}$; if $r > \gamma \delta^{-1}$
where $\gamma$ is a moderate constant (in our case we set $\gamma=10$), 
then we propose to
discard the solution and consider the 
matrix 
$\left[ \begin{smallmatrix} A_{11} & A_{12} \\ & A_{22} \end{smallmatrix} \right]$ as an atomic block.

However, this is not always viable because it could deteriorate the
spectral clustering property of the blocks, making the method based on 
Taylor expansions not efficient. In contrast, the approach based 
on randomized diagonalization applies with no modifications, although 
high precision arithmetic has to be employed on a larger block, causing 
an increase in the computational cost.

In our implementation, the merging of the blocks is adopted only when
relying on \textsc{fun2m\_atom\_diag}; the potential accuracy loss of 
the Taylor approach without merging is visible in the first example in 
Section~\ref{sec:numerical}.

\subsubsection{Complexity}

 We now prove that, under Assumption~\ref{ass:splitting}, the cost of Algorithm~\ref{alg:fun2m} is $\mathcal O(m^3+n^3)$ independently on the splitting choice. 

\begin{lemma}\label{lem:complexity}
Let $A\in\mathbb R^{m\times m}$, $B\in\mathbb  R^{n\times n}$, $C\in\mathbb R^{m\times n}$ and $f\{A,B\}(C)$ be computed by means of \textsc{fun2m} with a splitting strategy satisfying Assumption~\ref{ass:splitting}.  If \textsc{fun2\_atom} applied with arguments of sizes $p\times p$ and $q\times q$ costs  $\mathcal O(\max\{p,q\}pq)$ then \textsc{fun2m} requires $\mathcal O(m^3+n^3)$ flops.
\end{lemma}
\begin{proof}
We remark that the complexity of Algorithm~\ref{alg:fun2m_diag} is dominated by the cost of the reduction to Schur forms of $A$ and $B$ (that requires $\mathcal O(m^3+n^3)$ flops), the calls to \textsc{fun2\_atom} and the solution of the Sylvester equations.

Let us denote by $m_j=|I_j^A|$ and $n_j=|I_j^B|$ the sizes of the atomic blocks. Algorithm~\ref{alg:fun2m} calls  \textsc{fun2\_atom} $\ell_A\cdot \ell_B$ times and each call costs $
O(\max\{m_i,n_j\}m_in_j)$, where $\ell_A,\ell_B$ are the number of blocks in $A$ and $B$ respectively. Then, the overall cost of these calls is
\begin{align*}
\sum_{i=1}^{\ell_A}\sum_{j=1}^{\ell_B}\max\{m_i,n_j\}m_in_j&\leq \sum_{i=1}^{\ell_A}\sum_{j=1}^{\ell_B}(m_i+n_j)m_in_j\\
&=n\sum_{i=1}^{\ell_A}m_i^2+m\sum_{j=1}^{\ell_B}n_j^2
\leq m^2n+mn^2.
\end{align*}

The Sylvester equations are solved in a preprocessing step separately for $A$ and $B$. We prove by induction on the depth $d_A$ of 
the partitioning tree associated with $A$ that the cost of solving all Sylvester equations is $\mathcal O(m^3)$. The result for $B$ (that gives $\mathcal O(n^3)$) is analogous. When $d_A=1$ there is no Sylvester equation to solve. When $d_A>1$, let us suppose that the first splitting yields $\mathcal I^A=\mathcal I^{A_1}\sqcup\mathcal I^{A_2}$ with\footnote{\label{foot:1}By a slight abuse of notation, 
we write $|\mathcal I^{A_i}|$ to denote the sum of the cardinality of the index sets in $\mathcal I^{A_i}$, and analogously for $\mathcal I^{B_j}$.} $|\mathcal I^{A_1}|=m_1$ and $|\mathcal I^{A_2}|=m_2$. Then, we have to solve one Sylvester equation of size $m_1\times m_2$ and the Sylvester equations arising from the subtrees of depth $d_A-1$ associated with index sets of cardinality $m_1,m_2$, respectively. Solving the  Sylvester equation costs $\mathcal O(m_1m_2\min\{m_1,m_2\})$; the induction step yields $\mathcal O(m_1
^3)$ and $\mathcal O(m_2^3)$ for the subtrees. Summing these contributions we get $\mathcal O(m_1^3+m_2^3+m_1m_2\min\{m_1,m_2\})\leq O((m_1+m_2)^3)=\mathcal O(m^3)$.
\end{proof}
Although Lemma~\ref{lem:complexity} ensures the same asymptotic complexity independently on the splitting strategy, different choices might be preferable based on the underlying computer architecture. We remark that the atomic blocks of the splitting procedure are determined by \textsc{blocking}; any feasible partitioning tree has nodes given by (ordered) union of such atomic index sets and the latter correspond to the leaf nodes.  
We describe two strategies for constructing a feasible tree:

\begin{description}
\item[balanced] Each node $\mathcal I_A$ is split as $\mathcal I^{A_1}\sqcup\mathcal{I}^{A_2}$ with subsets $\mathcal I^{A_i}$ 
of approximately the same cardinality\cref{foot:1}.
\item[single] If $\mathcal I_A$ is composed by the atomic blocks $I^{A}_1,\dots,I^{A}_{\ell_A}$ then we consider the splitting $\mathcal I^{A_1}=\{I^{A}_1,\dots,I^{A}_{\ell_A-1}\}$, $\mathcal{I}^{A_2}=\{ I^{A}_{\ell_A} \}$.
\end{description}
In the numerical experiments in Section~\ref{sec:numerical} we adopt 
the \textbf{balanced} approach. The \textbf{single} approach is used in 
Section~\ref{ref:bs-connection} to discuss the connection with the 
Bartels--Stewart algorithm. We remark that both choices 
satisfy Assumption~\ref{ass:splitting}, and hence provide a cubic 
algorithm. 

	\subsection{Evaluating the function when one between $A$ and $B$ is small}
We conclude with a discussion about the evaluation 
of $f\{A,B^T\}(C)$ when  $m\gg n$, so that the outcome is a 
tall and thin matrix; 
the case $n\gg m$ is analogous. Similar considerations can be found also in 
\cite{kressner2011bivariate}.
In the case $n=1$, the problem reduces to computing a univariate function of the triangular matrix $A$ multiplied by the vector $C$. This can be done by relying on a Krylov method \cite{guttel2010rational}, and the cost depends on performing matrix-vector operations with $A$. When $n>1$, we consider the following cases:
\begin{itemize}
	\item[(i)]The eigenvalues of $B$ are clustered around $y_0$, so that we can find a low-degree univariate Taylor approximant of $f(x,y)\approx \sum_{j=0}^k \frac{\partial^j f(x, y_0)}{\partial y^j}\frac{(y-y_0)^j}{j!}$ centered at $y_0$. Then
	$$
	f\{A,B^T\}(C)\approx \sum_{j=0}^k 
	g_j(A)C(B-y_0I)^j,\qquad g_j(x):=\frac 1{j!}\frac{\partial^j f(x, y_0)}{\partial y^j}
	$$
	and the problem is recast as computing univariate functions of $A$ times (block) vectors and multiplications by (shifted) $B$.
	\item[(ii)] If the eigenvalues are not clustered as in (i), then we
	foresee two options. The first one is to block partition $B$, as described in Section~\ref{sec:blocking}, in order to retrieve the property on its atomic blocks. Finally, apply the same strategy as 
	in Algorithm~\ref{alg:fun2m} block-wise. The second is to 
	perturb and diagonalize $A$, and then use the formula 
	\eqref{eq:half-diagonalize}, to evaluate the univariate matrix
	functions at a higher precision in order to compensate for the 
	condition number of the eigenvector matrix of $A$. 
\end{itemize}
Note that, thanks to the triangular structure of $A$, rational Krylov subspace methods have the same asymptotic iteration cost of the standard polynomial Krylov method for the evaluation of $f(A)b$. 
Hence, unless a specific choice of shift parameters is known in advance (e.g., if a good rational approximant is known) the \emph{Extended Krylov method} might be a good choice. 

\section{Relation with other approaches}\label{sec:bartels}

Algorithm~\ref{alg:fun2m} is closely related with other known approaches 
for evaluating functions of matrices. In this section, we point out some
of these connections and differences.

\subsection{Recursive block diagonalization}
\label{sec:block-diagonalization}

The presented algorithm may be alternatively described 
avoiding Theorem~\ref{thm:splitting} as 
a recursive block diagonalization procedure, where the similarity
transformations are kept implicit. Indeed, given block triangular 
matrices $A$ and $B$, $V$ and $W$ as in Theorem~\ref{thm:splitting}, 
we have 
\[
  \underbrace{\begin{bmatrix}
    I & V \\
    & I \\
  \end{bmatrix}}_{\tilde V}
  f\{ A, B^T \}(C) 
  \underbrace{\begin{bmatrix}
    I & -W \\
    & I \\
  \end{bmatrix}}_{\tilde W^{-1}} = 
  f\left\{ 
    \begin{bmatrix}
      A_{11} \\
      & A_{22}
    \end{bmatrix}, \begin{bmatrix}
      B_{11}^T \\
      & B_{22}^T
    \end{bmatrix}
  \right\}(\tilde V C \tilde W^{-1}).
\]
Multiplying on the left by $\tilde V^{-1}$ and on the right by
$\tilde W$ and working out the relations on the blocks 
yields the same recursion obtained in 
Theorem~\ref{thm:splitting}. 

Working with transformations of this type allows to maintain the
diagonal blocks, relying on the blocking procedure to 
have well-conditioned Sylvester equations. 

\subsection{Algorithm~\ref{alg:fun2m} and the Bartels--Stewart algorithm} \label{ref:bs-connection}
In this section we will see that the celebrated Bartels--Stewart algorithm \cite{bartels1972algorithm} for 
solving Sylvester equations is closely related  to a particular case of Algorithm~\ref{alg:fun2m} applied to the function $f(x, y)=\frac{1}{x+y}$. We start by illustrating the relation between
block diagonalization and the backsubstitution method for
solving a triangular linear systems; then, we show that Algorithm~\ref{alg:fun2m} and the Bartels--Stewart algorithm verify 
a bivariate version of the latter relation.

Given a triangular matrix $A$, let us consider the linear system $Ax=b$ partitioned as
$$
\begin{bmatrix}
A_{11}&A_{12}\\
&A_{22}
\end{bmatrix}\begin{bmatrix}
x_1\\
x_2
\end{bmatrix}=\begin{bmatrix}
b_1\\ b_2
\end{bmatrix}.
$$
Applying the (block) backsubstitution procedure means to first compute $x_2=A_{22}^{-1}b_2$ and then $x_1= A_{11}^{-1}(b_1-A_{12}x_2)$. On the other hand, we might compute $A^{-1}b$ by first applying the similarity transformation
\[
  \underbrace{\begin{bmatrix}
    I & V \\
    & I 
  \end{bmatrix} \begin{bmatrix}
    A_{11} & A_{12} \\ 
    & A_{22} 
  \end{bmatrix} \begin{bmatrix}
    I & -V \\
    & I 
  \end{bmatrix}}_{= \left[ \begin{smallmatrix}
   A_{11} \\ & A_{22}
  \end{smallmatrix} \right]}
  \begin{bmatrix}
    x_1 + Vx_2 \\ 
    x_2 
  \end{bmatrix} = 
  \begin{bmatrix}
    b_1 + Vb_2 \\
    b_2
  \end{bmatrix},
\]
where $A_{11} V - VA_{22} = A_{12}$. This approach would result in the following steps:
\begin{enumerate}
    \item Compute $x_2=A_{22}^{-1}b_2$,
    \item Compute $\widetilde x_1=A_{11}^{-1}(b_1 + Vb_2)$,
    \item Compute $x_1= \widetilde x_1-  Vx_2$. 
\end{enumerate}
In particular, both procedures solve two triangular systems whose coefficient matrices are the diagonal blocks of $A$. However, the one based on block-diagonalization needs corrections that require the solution of $A_{11}V-VA_{22}=A_{12}$. In particular, 
the block-diagonalization requires a cubic cost, whereas 
back-substitution is quadratic. 

Let us recall the procedure by Bartels and Stewart by using the notation introduced in Section~\ref{sec:main}. Given the Sylvester equation $AX+XB=C$ with $A$ and $B$ upper triangular (possibly after the computation of the Schur forms) we consider the partitioning
\begin{equation}\label{eq:bartels}
 \begin{bmatrix}
 A_{11}&A_{12}\\
 &A_{22}
 \end{bmatrix}
 \begin{bmatrix}
 X_{11}&X_{12}\\
 X_{21}&X_{22}
 \end{bmatrix}+
 \begin{bmatrix}
 X_{11}&X_{12}\\
 X_{21}&X_{22}
 \end{bmatrix}\begin{bmatrix}
 B_{11}&B_{12}\\
 &B_{22}
 \end{bmatrix}=\begin{bmatrix}
 C_{11}&C_{12}\\
 C_{21}&C_{22}
 \end{bmatrix}
\end{equation}
where $A_{22}$ and $B_{11}$ are scalars. The Bartels--Stewart algorithm retrieves the blocks $X_{ij}$ as follows:
\begin{enumerate}
    \item solve the scalar equation associated to the $(2,1)$ block: $X_{21}= \frac{C_{21}}{A_{22}+B_{11}}$,
    \item solve the triangular linear system $(A_{11}+ B_{11}I)X_{11} = C_{11}-A_{12}X_{21} $
    \item solve the triangular linear system $X_{22}(B_{22}+ A_{22}I) = C_{22}-X_{21}B_{12} $
    \item recursively solve the Sylvester equation $$A_{11}X_{12}+X_{12}B_{22} = C_{12}-A_{12}X_{22}-X_{12}B_{12}.$$
\end{enumerate}
We now analyze the relation between  the previous steps and the four quantities in Theorem~\ref{thm:splitting}, applied with $k_A = 1$ and $k_B=n-1$. Given $f(x,y)=\frac{1}{x+y}$, we remark that when at least one of the arguments of the bivariate matrix function $f\{A,B\}$ is a scalar the associated operator is the resolvent of a scalar equation or a linear system. In our setting we have
$$
    f\{A_{22},B_{11}^T\}(C_{21})=\frac{C_{21}}{A_{22}+B_{11}},$$
    that is equivalent to step 1. of Bartels--Stewart and
    \begin{align*}
     f\{A_{11}, B_{11}^T\}(C_{11} + VC_{21})&=(A_{11}+B_{11}I)^{-1}(C_{11}+VC_{21}),\\
   f\{A_{22}, B_{22}^T\}(C_{22}-C_{21}W)&=(C_{22}-C_{21}W)(B_{22}+A_{22}I)^{-1},
\end{align*}
where the column and row vectors $V$ and $W$ are given by
\begin{align*}
    V&=(A_{11}-A_{22}I)^{-1}A_{12},\\
    W&=B_{12}(B_{11}I-B_{22})^{-1}.
\end{align*}
Then, $X_{11}$ and $X_{22}$ are computed by
\[
  X_{11} = f\{A_{11},B_{11}^T\}(C_{11} + VC_{21}) - VX_{21}, \quad 
  X_{22} = f\{A_{22}, B_{22}^T\}(C_{22}-C_{21}W) - X_{21} W.
\]
Finally, the $X_{12}$ is computed recursively by:
\begin{align*}
X_{12} &= f\{A_{11},B_{22}^T\}\left(\begin{bmatrix}
	I& V
	\end{bmatrix}C \begin{bmatrix}
	-W\\ I
	\end{bmatrix}\right) \\
	&+(X_{11}+VX_{21})W -VX_{21}W- V(X_{22}+X_{21}W) 
\end{align*}
and $X_{12}$ is computed by removing the three rank one corrections.
In contrast with the univariate case, in this case both approaches have 
the same asymptotic complexity, even though Bartels--Stewart is more
efficient, since it does not need to apply the corrective terms. 

However, this further optimization is viable only because of the special
features of $f(x,y) = \frac{1}{x+y}$. Indeed, 
one can verify that 
\[
  f\{ A_{11}, B_{11}^T \}(C_{11} + VC_{21}) + V X_{21} = 
    f\{ A_{11}, B_{11}^T \} (C_{11} - A_{12} X_{21})
\]
only hold for $f(x,y) = \frac{1}{x+y}$, and similarly for the 
relations for $X_{22}$ and $X_{12}$. These are the key properties 
that allow to avoid computing the term $V$ and $W$ explicitly in the Bartels--Stewart algorithm.

\section{Numerical results}\label{sec:numerical}
In this section we test the performances of \textsc{fun2m} and of the various choices that can be made in its implementation for computing $f\{A,B^T\}(C)$. We note that the choice of the matrix $C$ does not affect the behavior of \textsc{fun2m}. Everywhere, we set $C$ equals to a \emph{random  complex matrix}; the latter indicates that both real and imaginary parts have $N(0,1)$-distributed entries, throughout this section. For simplicity we also assume $m=n$ in all our tests. Concerning the choice of $A$ and $B$ we introduce the following test cases: 
\begin{description}
\item[\texttt{rand-eig}] Both $A$ and $B$ are of the form $VDV^{-1}$ where $D$ is a random diagonal matrix  whose entries have a real part uniformly distributed on $[1,2]$ and Gaussian distributed imaginary parts; the matrix $V$ is a random complex matrix with both real and imaginary part of its entries Gaussian distributed.
    \item[\texttt{randn}]  Both $A$ and $B$ are complex random matrices.    
  \item[\texttt{jordbloc}]   Both $A$ and $B$ are of the form $QJQ^*$ where Q is a random unitary matrix (obtained by means of the QR factorization of a random complex matrix) and $J$ is the 
  direct sum of a $8\times 8$ Jordan block with eigenvalue $0.1$ and a complex random matrix of size $n-8$ shifted by the identity; $B$ is generated analogously.
 \item[\texttt{grcar}] $A$ and $B$ are equal to the \texttt{grcar} matrix of the Matlab gallery.
\item[\texttt{smoke}] $A$ and $B$ are equal to the Schur form of the \texttt{smoke} matrix of the Matlab gallery.
  \item[\texttt{kahan}] $A$ and $B$ are equal to the \texttt{kahan} matrix of the Matlab gallery.
  \item[\texttt{lesp}] The matrices $-A$ and $-B$ are equal to the direct sum of the Schur form of the \texttt{lesp} matrix of the Matlab gallery of dimension $32$ with a random matrix of size $n-32$. The latter is obtained by generating a complex random matrix, rescaling it to have unit spectral norm and subtracting  the identity.
  \item[\texttt{sampling}] The matrices $A$ and $B$ are equal to the direct sum of the \texttt{sampling} matrix of the Matlab gallery of dimension $32$ with a random matrix of size $n-32$. The latter is obtained by generating a complex random matrix, rescaling it to have unit spectral norm and adding  the identity.
  \item[\texttt{grcar-rand}] $A$ is equal to the \texttt{grcar} matrix of the Matlab gallery and $B$ is as in \texttt{rand-eig}. 
\end{description}
When experimenting with  \textsc{fun2m} we indicate in bracket the method used for evaluating the atomic blocks, i.e. \textsc{fun2\_atom\_diag} or \textsc{fun2\_atom\_taylor}. 
The other considered computational approaches are labeled as follows:
\begin{description}
\item[\normalfont\textsc{diag}] The evaluation of $f\{A,B^T\}(C)$ is performed diagonalizing $A$ and $B$ in floating point arithmetic, regardless of the conditioning of the eigenvector matrices.
\item[\normalfont\textsc{diag\_hp}] The evaluation of $f\{A,B^T\}(C)$ is performed diagonalizing $A$ and $B$ in high precision, estimating the required digits as in \textsc{fun2\_atom\_diag}.
\end{description}
We expect the first method to be fast with no guarantee on its accuracy. In contrast, the second approach is the most accurate although it can be significantly more expensive than both \textsc{diag} and \textsc{fun2m}. In the tables, the columns labeled as $n_A$ and $n_B$ denote the number of atomic blocks in $A$ and $B$ respectively. The label ``Digits'' refers to the maximum number of digits used in the multiprecision computation of the functions of the atomic blocks. The column 
``Max deg'' contains the maximum of the degrees of the Taylor expansions used for the atomic blocks. Residual errors in the spectral norm are evaluated with respect to a benchmark quantity computed as in \textsc{diag\_hp} where the number of digits is fixed to $128$. The latter value is in all cases much higher than the number of digits employed by \textsc{fun2m} and \textsc{diag\_hp}.
Finally, for each example we provide a very rough estimate $\kappa_f$ of the condition number of 
evaluating $f\{A,B^T\}(C)$. The latter is defined as
$$
\lim_{h\to 0} \sup_{\frac{\norm{\Delta A}}{\norm {A}},\frac{\norm{\Delta B}}{\norm{B}}\leq h}\frac{\norm{f\{A+\Delta A,B^T+\Delta B^T\}(C)-f\{A,B^T\}(C)}}{h}.
$$
We compute $\kappa_f$ by evaluating the above fraction in higher precision ($128$ digits) for $h=10^{-32}$ and
random complex matrices $\Delta A,\Delta B$ scaled to have norm $h\norm{A}$ and $h\norm{B}$, respectively. 
This estimate is quite rough in general, but it yields a guaranteed lower bound, and usually captures the order 
of magnitude of the condition number, which is sufficient to assess the accuracy of our results. 
In the tables we report $\kappa_f \cdot u$, where $u$ is the unit round-off in double precision, which gives an indication of the accuracy attainable by a backward stable method.

The algorithms have been implemented in a Julia package named \texttt{BivMatFun} and available at \url{https://github.com/numpi/BivMatFun}. 
The implementation of \textsc{fun2m} may be further optimized relying on the recursive Sylvester triangular solver \texttt{recsy} \cite{jonsson2003recsy} 
and the BLAS 3 reordering of the Schur form in~\cite{kressner2006block}. For simplicity we have used what is available in LAPACK, through the 
interfaces in Julia. 
The experiments have been run using Julia 1.5.3 on a dual CPU server with two
Intel(R) Xeon(R) CPU E5-2643 v4 CPUs running at 3.40GHz, and 240GB of RAM.

\subsection{Perturbation and diagonalization versus Taylor expansion}
In the first numerical test we compare the two proposed implementation for the functions of the atomic blocks, i.e. \textsc{fun2\_atom\_taylor} and \textsc{fun2\_atom\_diag}. We have tested two cases: \texttt{rand-eig} and \texttt{grcar-rand}. In the first, both eigenvector matrices are sufficiently well conditioned and the accuracy achieved by the methods is similar.  Moreover the blocking procedure allows to form atomic blocks of small sizes, so that the cost of employing high precision arithmetic does not impact at all. Indeed, the timings are in favour of \textsc{fun2\_atom\_diag}. In the second test case, the eigenvector matrix of $A$ is severely ill-conditioned and this is reflected in the magnitude of the solutions of the Sylvester equations computed in Algorithm
~\ref{alg:fun2m}. When using \textsc{fun2\_atom\_diag} this issue is circumvented by merging the blocks of $A$ into a single one, at the price of an increased computational cost caused by the use of higher precision arithmetic on larger matrices. This procedure can not be applied by \textsc{fun2\_atom\_taylor} because this would cause a lack of the convergence for the Taylor expansion of $\frac{1}{\sqrt{x+y}}$. Hence, the timings of the Taylor-based approach are similar but the outcome is not reliable. These remarks are confirmed by the results reported in Figure~\ref{fig:exp1}.

Finally, we mention that when estimating the eigenvector condition number for \texttt{grcar} and $n=160$ the heuristic estimate based on Equation~\ref{eq:kestimate} fails and our procedure detects this with the a posteriori check and repeats the computation of the eigenvector with the appropriate accuracy.
 \begin{figure}
 	\centering
\resizebox{\textwidth}{!}{ 
 	\begin{filecontents*}{data/exp1-mat1.dat}
  32.0  4.11374e-15  9.19418e-15   8.0   8.0   30.0   32.0  18.0   7.0  9.19523e-16  0.00519023  0.0135706
  64.0  9.2402e-15   9.65174e-15  18.0  17.0   57.0   62.0  18.0   9.0  1.16928e-14  0.02467     0.080806
  96.0  3.96863e-13  3.97117e-13  32.0  32.0   88.0   85.0  19.0  10.0  1.16661e-12  0.0710787   0.224521
 128.0  5.69956e-13  5.70191e-13  37.0  36.0  106.0  112.0  19.0  11.0  2.56764e-13  0.132409    0.425397
 160.0  9.91632e-15  1.87131e-14  59.0  60.0  129.0  140.0  18.0  11.0  4.46511e-15  0.200049    0.614876

\end{filecontents*}
\pgfplotstabletypeset[
 	every head row/.style={
 		before row={
 			\multicolumn{12}{c}{Test $=$ \texttt{randn-shift}, $f(x,y)=\frac 1{\sqrt{x+y}}$}\\[1ex] \toprule
 			&\multicolumn{5}{c|}{\textsc{fun2m} (\textsc{fun2\_atom\_diag})} &\multicolumn{5}{c|}{\textsc{fun2m} (\textsc{fun2\_atom\_taylor})}
 			& \\
 		},
 		after row = \midrule,
 	},
 	columns = {0,1,10,3,4,7,2,11,5,6,8,9},
 		columns/0/.style = {column name = Size},
 	columns/1/.style = {column name = Err,precision=1,zerofill, column type=|c},
 	columns/10/.style = {column name = Time, fixed},
 	columns/3/.style = {column name = nA},
 	 columns/4/.style = {column name = nB},
 	columns/7/.style = {column name = Digits, column type=c},
  	columns/2/.style = {column name = Err,precision=1,zerofill, column type=|c},
 	columns/11/.style = {column name = Time, fixed},
 	columns/5/.style = {column name = nA},
 	 columns/6/.style = {column name = nB},
 	columns/8/.style = {column name = Max deg, column type=c},
 	 	columns/9/.style = {column name = $\kappa_f\cdot u$,precision=1,zerofill, column type = |c}
 	]{data/exp1-mat1.dat}}
  \vspace{.2cm}

\resizebox{\textwidth}{!}{ 
 	\begin{filecontents*}{data/exp1-mat2.dat}
  32.0  2.9895e-15   9.29918e-12  1.0   8.0  32.0   32.0  22.0   0.0  3.87528e-16  0.093416  0.0118429
  64.0  4.44593e-15  0.000360927  1.0  17.0  56.0   62.0  28.0   9.0  1.01574e-15  0.567242  0.0790758
  96.0  7.03923e-15  3.1912e5     1.0  32.0  68.0   85.0  35.0  13.0  2.47962e-15  1.66951   0.220035
 128.0  9.67579e-15  5.66976e9    1.0  36.0  74.0  112.0  36.0  16.0  4.82756e-15  3.83198   0.409644
 160.0  8.13669e-15  1.67138e28   1.0  60.0  61.0  140.0  38.0  22.0  1.5936e-15   6.93378   0.651048

\end{filecontents*}
\pgfplotstabletypeset[
 	every head row/.style={
 		before row={
 			\multicolumn{12}{c}{Test $=$ \texttt{grcar-rand}, $f(x,y)=\frac 1{\sqrt{x+y}}$}\\[1ex] \toprule
 			&\multicolumn{5}{c|}{\textsc{fun2m} (\textsc{fun2\_atom\_diag})} &\multicolumn{5}{c|}{\textsc{fun2m} (\textsc{fun2\_atom\_taylor})}
 			& \\
 		},
 		after row = \midrule,
 	},
 	columns = {0,1,10,3,4,7,2,11,5,6,8,9},
 		columns/0/.style = {column name = Size},
 	columns/1/.style = {column name = Err,precision=1,zerofill, column type=|c},
 	columns/10/.style = {column name = Time, fixed},
 	columns/3/.style = {column name = nA},
 	 columns/4/.style = {column name = nB},
 	columns/7/.style = {column name = Digits, column type=c},
  	columns/2/.style = {column name = Err,precision=1,zerofill, column type=|c},
 	columns/11/.style = {column name = Time, fixed},
 	columns/5/.style = {column name = nA},
 	 columns/6/.style = {column name = nB},
 	columns/8/.style = {column name = Max deg, column type=c},
 	 	columns/9/.style = {column name = $\kappa_f\cdot u$,precision=1,zerofill, column type = |c}
 	]{data/exp1-mat2.dat}}
\caption{Performances of the diagonalize-and-perturb and of the bivariate Taylor approximation on well-conditioned and ill-conditioned test cases.}\label{fig:exp1}
\end{figure}

\subsection{Highly non normal $A$ and $B$}
Here, we consider seven test cases of fixed size $n=64$, that involve both $A$ and $B$ with ill-conditioned eigenvector matrices. We compare the performances of \textsc{fun2m} with \textsc{diag} and \textsc{diag\_hp} on four different bivariate functions: $\sqrt{x+y}, \frac{1}{\sqrt{x+y}}, \frac{\exp(x+y)}{x+y}, \exp(\sqrt{x+y})$. In view of the considerations made in the previous experiment we only rely on \textsc{fun2\_atom\_diag} for evaluating the atomic blocks  in \textsc{fun2m}.  The results reported in Figure~\ref{fig:exp2} confirm that \textsc{diag} is the fastest and least reliable method. On the other hand, \textsc{fun2m} and \textsc{diag\_hp} are equally accurate with \textsc{fun2m} outperforming \textsc{diag\_hp} apart from  the cases where the partitioning is trivial --- $n_A=n_B=1$ --- where the two algorithms coincide.
We note that in many cases the residual obtained by \textsc{fun2m} and \textsc{diag\_hp} is significantly below the estimate given by $\kappa_f\cdot u$. This is motivated by the fact that the matrices are (partially) upper triangular and the condition number with respect to perturbations sharing the same sparsity pattern is smaller.

	 	\pgfplotstableset{ 	create on use/Test/.style={string type,
    create col/set list={\texttt{jordbloc}, \texttt{grcar}, \texttt{smoke}, \texttt{kahan}, \texttt{lesp}, \texttt{sampling}, \texttt{grcar-rand}}},
columns/Test/.style={string type}}
 \begin{figure}
 	\centering
\resizebox{\textwidth}{!}{ 
 	\begin{filecontents*}{data/exp2-1.dat}
7.87189707013824e-10	1.9510671146116638	7.871884516963004e-10	0.0196083125	0.008049102	1.63535607	15.0	15.0	48.0	50.0	1.0792801593489865e-10
1.1306224362280828e-13	4.4615977640324316e-7	1.13066466802369e-13	1.4985379275	0.007315222	1.4873269835	1.0	1.0	40.0	40.0	3.408744223391283e-10
8.160817021116375e-14	1.2772684464218997e-8	6.492960105661613e-14	1.534039815	0.0014106095	1.4233614505	1.0	1.0	35.0	35.0	2.27554751267706e-5
2.497522087565807e-16	0.0003966403330427437	6.644547552670019e-16	1.3540676775	0.001315135	1.382759355	1.0	1.0	43.0	43.0	2.3657341765179833e-8
2.5517017022821566e-15	1.422534258720289	2.2810120823545384e-15	0.2277300955	0.0030357975	1.29486942	9.0	9.0	35.0	36.0	2.604126221291807e-16
4.509509136443055e-8	0.3547851125518197	4.5095090836641266e-8	1.445249663	0.0054409525	2.173788625	1.0	9.0	49.0	49.0	9.706562158382151e-9
1.6172296366326102e-12	1.227571110665487e-7	1.6170042932288008e-12	0.401413099	0.00847889	1.4784844705	1.0	16.0	29.0	31.0	5.436395168027492e-7
\end{filecontents*}
\pgfplotstabletypeset[
 	every head row/.style={
 		before row={
 			\multicolumn{12}{c}{$f(x,y)=\sqrt{x+y}$, Size $= 64$}\\[1ex] \toprule
 			&\multicolumn{5}{c|}{\textsc{fun2m} (\textsc{fun2\_atom\_diag})} &\multicolumn{2}{c|}{\textsc{diag}}
 			&\multicolumn{3}{c|}{\textsc{diag\_hp}}& \\
 		},
 		after row = \midrule,
 	},
 columns = {Test,0,3,6,7,8,1,4,5,2,9,10},
 columns/0/.style = {column name = Err,precision=1,zerofill, column type=|c},
 columns/3/.style = {column name = Time, fixed},
 columns/6/.style = {column name = nA},
  columns/7/.style = {column name = nB},
 columns/8/.style = {column name = Digits, column type=c|},
 columns/1/.style = {column name = Err,precision=1,zerofill, sci},
 columns/4/.style = {column name =Time, fixed ,fixed,precision =3, column type =c|},
columns/5/.style = {column name = Time, fixed},
columns/2/.style = {column name =Err,precision=1,zerofill},
columns/9/.style = {column name = Digits, column type = c|},
columns/10/.style = {column name = $\kappa_f\cdot u$,precision=1,zerofill}
 	]{data/exp2-1.dat}}
  \vspace{.2cm}
 
\resizebox{\textwidth}{!}{ 
 	\begin{filecontents*}{data/exp2-2.dat}
2.014349642039755e-9	0.3869179509987958	2.0143488911271775e-9	0.0194959015	0.008295386	1.648578101	15.0	15.0	48.0	51.0	3.217204961579436e-10
1.4675206884279714e-13	7.679035988954096e-8	1.4671314058856009e-13	1.527638097	0.007511561	1.54182259	1.0	1.0	40.0	40.0	1.0085565442746417e-9
3.5118480875096193e-9	1.0873046767269688e-8	1.798329416978147e-9	1.4592417635	0.001672663	1.4521389615	1.0	1.0	35.0	35.0	0.49549953859664114
3.3763633273219856e-16	6.750363472683712e-7	4.541894990671936e-16	1.370440367	0.001550789	1.3639941155	1.0	1.0	43.0	43.0	1.369712670095717e-7
4.3550583085174745e-15	0.1601640093472295	3.4517785733448327e-15	0.233248588	0.003266952	1.316344476	9.0	9.0	35.0	36.0	1.8570505590455035e-15
1.0298242459007029e-7	0.021793664002906773	1.0298242386433427e-7	0.405021242	0.0055859235	2.041396426	10.0	9.0	49.0	49.0	8.207551177276277e-8
5.2150992646600936e-12	7.847227751347471e-8	5.214879138293369e-12	0.3881454135	0.008723239	1.4534485505	1.0	16.0	29.0	31.0	3.7215759257725295e-6
\end{filecontents*}
\pgfplotstabletypeset[
 	every head row/.style={
 		before row={
 			\multicolumn{12}{c}{$f(x,y)=\frac{1}{\sqrt{x+y}}$, Size $=64$}\\[1ex] \toprule
 			&\multicolumn{5}{c|}{\textsc{fun2m} (\textsc{fun2\_atom\_diag})} &\multicolumn{2}{c|}{\textsc{diag}}
 			&\multicolumn{3}{c|}{\textsc{diag\_hp}}& \\
 		},
 		after row = \midrule,
 	},
 columns = {Test,0,3,6,7,8,1,4,5,2,9,10},
 columns/0/.style = {column name = Err,precision=1,zerofill, column type=|c},
 columns/3/.style = {column name = Time, fixed},
 columns/6/.style = {column name = nA},
  columns/7/.style = {column name = nB},
 columns/8/.style = {column name = Digits, column type=c|},
 columns/1/.style = {column name = Err,precision=1,zerofill, sci},
 columns/4/.style = {column name =Time, fixed ,fixed,precision =3, column type =c|},
columns/5/.style = {column name = Time, fixed},
columns/2/.style = {column name =Err,precision=1,zerofill},
columns/9/.style = {column name = Digits, column type = c|},
columns/10/.style = {column name = $\kappa_f\cdot u$,precision=1,zerofill}
 	]{data/exp2-2.dat}}
 	\vspace{.2cm}
 
\resizebox{\textwidth}{!}{ 
 	\begin{filecontents*}{data/exp2-3.dat}
1.1693294556917817e-14	2.473978594851756e9	8.931288391509987e-15	0.020636866	0.0084567315	1.671759811	17.0	15.0	48.0	50.0	2.2663420176613264e-16
7.863594457544239e-15	692.7061194612494	7.864031972926324e-15	1.526897522	0.0075231325	1.5327602385	1.0	1.0	40.0	40.0	3.931783562942412e-16
4.9140366303473865e-17	0.8445743933343629	4.8678010690293743e-17	1.48753588	0.0016811905	1.477074622	1.0	1.0	35.0	35.0	2.3376716242191285e-16
4.736862425810669e-17	2.1790385557875507e7	4.7343369001993e-17	1.335380737	0.001473083	1.354223791	1.0	1.0	43.0	43.0	1.6857787850027677e-16
4.5124432899605753e-17	2.413363524614756e-13	4.5124419095776784e-17	0.2309848075	0.0031113605	1.326773175	9.0	10.0	35.0	36.0	2.2540124567610804e-13
1.6452506095008532e-8	0.0489300321977261	1.64525060503707e-8	1.3989401955	0.005643087	2.094252804	1.0	9.0	49.0	49.0	1.2761608389053277e-9
1.813278866639911e-14	4.831142504682691e-7	1.811852200973512e-14	0.362874843	0.0085314735	1.503831926	1.0	16.0	29.0	31.0	6.957385109276823e-16
\end{filecontents*}
\pgfplotstabletypeset[
 	every head row/.style={
 		before row={
 			\multicolumn{12}{c}{$f(x,y)=\frac{\exp(x+y)}{x+y}$, Size $=64$}\\[1ex] \toprule
 			&\multicolumn{5}{c|}{\textsc{fun2m} (\textsc{fun2\_atom\_diag})} &\multicolumn{2}{c|}{\textsc{diag}}
 			&\multicolumn{3}{c|}{\textsc{diag\_hp}}& \\
 		},
 		after row = \midrule,
 	},
 columns = {Test,0,3,6,7,8,1,4,5,2,9,10},
 columns/0/.style = {column name = Err,precision=1,zerofill, column type=|c},
 columns/3/.style = {column name = Time, fixed},
 columns/6/.style = {column name = nA},
  columns/7/.style = {column name = nB},
 columns/8/.style = {column name = Digits, column type=c|},
 columns/1/.style = {column name = Err,precision=1,zerofill, sci},
 columns/4/.style = {column name =Time, fixed ,fixed,precision =3, column type =c|},
columns/5/.style = {column name = Time, fixed},
columns/2/.style = {column name =Err,precision=1,zerofill},
columns/9/.style = {column name = Digits, column type = c|},
columns/10/.style = {column name = $\kappa_f\cdot u$,precision=1,zerofill}
 	]{data/exp2-3.dat}}
 \vspace{.2cm}
 
\resizebox{\textwidth}{!}{ 
 	\begin{filecontents*}{data/exp2-4.dat}
2.0486491082056408e-10	7.585073559534292	2.048644380182821e-10	0.0235575295	0.008500498	1.767986268	15.0	15.0	48.0	51.0	3.8776752613166723e-10
1.1217211083194005e-13	1.4946614829495298e-6	1.1225570487896263e-13	1.61468954	0.0077115635	1.5986566245	1.0	1.0	40.0	40.0	6.916401245260714e-10
1.660673823928913e-13	2.0773176920817375e-8	1.086971425653914e-13	1.5597651835	0.001856961	1.540688458	1.0	1.0	35.0	35.0	3.9747219715168745e-5
1.4237777786813556e-14	0.010043111247481053	1.3479364922633431e-14	1.4343275335	0.00131986	1.3605843325	1.0	1.0	43.0	43.0	5.508804043402153e-7
2.4714729610443044e-16	0.41316775998417576	2.43951613875374e-16	0.233499327	0.0033500515	1.4043406725	9.0	9.0	35.0	36.0	2.3183385939548653e-16
4.971691819782023e-8	381.7749091630773	4.9716918322814806e-8	2.186994824	0.005667651	2.173970016	1.0	1.0	49.0	49.0	9.98842991866539e-9
6.308050703543101e-13	8.452605653164599e-8	6.311459287367435e-13	0.40178835	0.008878325	1.574571408	1.0	16.0	29.0	31.0	1.5600681685761295e-5
\end{filecontents*}
\pgfplotstabletypeset[
 	every head row/.style={
 		before row={
 			\multicolumn{12}{c}{$f(x,y)=\exp(\sqrt{x+y})$, Size $= 64$}\\[1ex] \toprule
 			&\multicolumn{5}{c|}{\textsc{fun2m} (\textsc{fun2\_atom\_diag})} &\multicolumn{2}{c|}{\textsc{diag}}
 			&\multicolumn{3}{c|}{\textsc{diag\_hp}}& \\
 		},
 		after row = \midrule,
 	},
 columns = {Test,0,3,6,7,8,1,4,5,2,9,10},
 columns/0/.style = {column name = Err,precision=1,zerofill, column type=|c},
 columns/3/.style = {column name = Time, fixed},
 columns/6/.style = {column name = nA},
  columns/7/.style = {column name = nB},
 columns/8/.style = {column name = Digits, column type=c|},
 columns/1/.style = {column name = Err,precision=1,zerofill, sci},
 columns/4/.style = {column name =Time, fixed ,fixed,precision =3, column type =c|},
columns/5/.style = {column name = Time, fixed},
columns/2/.style = {column name =Err,precision=1,zerofill},
columns/9/.style = {column name = Digits, column type = c|},
columns/10/.style = {column name = $\kappa_f\cdot u$,precision=1,zerofill}
 	]{data/exp2-4.dat}}
 	\caption{Numerical results on highly non normal matrices of sizes $n=m=64$.}
 	\label{fig:exp2}
 \end{figure}

 \subsection{Asymptotic cost}
 In this final example we test the computational cost of \textsc{fun2m} on well-conditioned test cases where the number of atomic blocks in $A$ and $B$ grows linearly with the size $n$. More specifically, we consider the test case \texttt{randn} with exponentially increasing sizes $2^j$ for $j=6,\dots, 12$ and we measure the computational time. The performances are compared with the ones of \textsc{diag} in Figure~\ref{fig:exp3}. The methods have comparable costs with \textsc{diag} being faster. Both approaches scales quadratically up to dimension $2048$ and we start to see the expected cubic growth only on the last test. We mention that the measured accuracies are comparable and since this is a well-conditioned case we refrain to report them.
 
 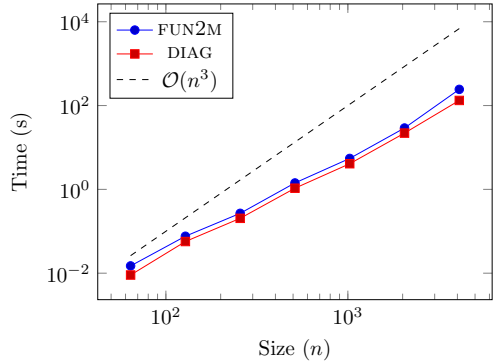
\begin{figure}
 \centering
 \begin{minipage}{.4\textwidth}
 \resizebox{\textwidth}{!}{ 

	\begin{filecontents*}{data/exp3.dat}
64.0	0.01492825	0.013254278	0.009040911	16.0	16.0
128.0	0.0762048875	0.072362646	0.056918267	32.0	32.0
256.0	0.268109491	0.2594763875	0.203671952	64.0	64.0
512.0	1.4235130215	1.4091439135	1.065753853	128.0	128.0
1024.0	5.473917331	5.416291782	4.0916401425	256.0	256.0
2048.0	29.121758091	28.03238288	22.010552033	512.0	512.0
4096.0	243.408700703	234.618018395	132.48330466	1024.0	1024.0
\end{filecontents*}
\pgfplotstabletypeset[
 	every head row/.style={
 		before row={
 			\multicolumn{5}{c}{Test $=$ \texttt{randn}, $f(x,y)=\frac 1{\sqrt{x+y}(x-y)}$, Size$=64$}\\[1ex] \toprule
 			&\multicolumn{3}{c|}{\textsc{fun2m} (\textsc{fun2\_atom\_diag})} &\textsc{diag}
 			 \\
 		},
 		after row = \midrule,
 	},
 	columns = {0,1,4,5,3},
 		columns/0/.style = {column name = Size},
 	columns/1/.style = {column name = Time,fixed,zerofill, column type=|c},
 	 	columns/4/.style = {column name = nA},
 	 columns/5/.style = {column name = nB},
 	columns/3/.style = {column name = Time, fixed,column type = |c}
 	]{data/exp3.dat}}
 	\rule{0pt}{1.2cm}
\end{minipage}~\begin{minipage}{.5\textwidth}
\begin{tikzpicture}[scale=.8]
 \begin{loglogaxis}[width=1.25*\textwidth, 
    height = 6.5cm, 
    xlabel = Size ($n$), ylabel = Time (s), legend pos = north west]
 \addplot table[x index = 0, y index = 1]{data/exp3.dat};
  \addplot table[x index = 0, y index = 3]{data/exp3.dat};
  \addplot[dashed, domain = 64:4096,samples = 2]{1e-7*x^3};
  \legend{\textsc{fun2m}, \textsc{diag}, $\mathcal O(n^3)$}
 \end{loglogaxis}
\end{tikzpicture}
 \end{minipage}
 \caption{Timings of \textsc{fun2m} and \textsc{diag} for well-conditioned $A$ and $B$.}\label{fig:exp3}
 \end{figure}
\section{Conclusions}\label{sec:conclusion}
We have proposed a novel block diagonalization approach for the evaluation of bivariate matrix functions. 
By relying on the synergy of multiprecision and a blocking strategy analogous to the one used 
in the Schur-Parlett scheme for univariate functions, the method guarantees backward stable results. 
We have validated the stability properties by testing the algorithm on a wide range of ill-conditioned cases.
The asymptotic complexity is $\mathcal O(m^3+n^3)$ where $m$ and $n$ correspond to the size of the two square 
matrix arguments, independently on the blocking strategy applied. 
In the ideal situation of well conditioned eigenvector matrices the performances are comparable to evaluating the 
function by diagonalization. The algorithm extends naturally to the multivariate case although the number of atomic
blocks grows exponentially with the number of variables. 
   
\bibliography{biblio}
\bibliographystyle{siamplain}
\end{document}